\documentclass[10pt]{siamltex}
\usepackage{epsfig}
\usepackage{amssymb}
\usepackage[bbgreekl]{mathbbol}
\usepackage{amsfonts}
\DeclareSymbolFontAlphabet{\mathbb}{AMSb}
\DeclareSymbolFontAlphabet{\mathbbl}{bbold}
\usepackage{amsbsy}
\usepackage{amsmath}
\usepackage{enumerate}
\usepackage{float}
\usepackage{graphicx}
\usepackage{xspace}
\usepackage{bm}
\usepackage{color}
\usepackage{diagrams}
\usepackage{remarks}
\usepackage{def} 
\usepackage[numbers]{natbib}
\def\bbeps{\bbespilon}

\begin{document}

\newcommand\footnotemarkfromtitle[1]{%
\renewcommand{\thefootnote}{\fnsymbol{footnote}}%
\footnotemark[#1]%
\renewcommand{\thefootnote}{\arabic{footnote}}}

\title{Mollification in strongly Lipschitz domains 
  with application to continuous and discrete\\ De Rham complexes\footnotemark[1]\ \footnotemark[4]}
\author{Alexandre Ern\footnotemark[2] \and Jean-Luc
  Guermond\footnotemark[3]}

\maketitle

\renewcommand{\thefootnote}{\fnsymbol{footnote}} \footnotetext[1]{
  This material is based upon work supported in part by the National
  Science Foundation grants DMS-1015984 and DMS-1217262, by the Air
  Force Office of Scientific Research, USAF, under grant/contract
  number FA99550-12-0358.  Draft
  version, \today}
\footnotetext[2]{Universit\'e Paris-Est, CERMICS (ENPC),
  77455 Marne-la-Vall\'ee cedex 2, France.}
\footnotetext[3]{Department of Mathematics, Texas
  A\&M University 3368 TAMU, College Station, TX 77843, USA}
\footnotetext[4]{Published: A. Ern, J.-L. Guermond,
Mollification in strongly Lipschitz domains 
  with application to continuous and discrete De Rham complex,
\textit{Computational Methods in Applied Mathematics},
Volume 16, Issue 1, (2016) 51--75.}

\renewcommand{\thefootnote}{\arabic{footnote}}

\begin{abstract} We construct mollification operators in strongly
  Lipschitz domains that do not invoke non-trivial extensions, are
  $L^p$ stable for any real number $p\in[1,\infty]$, and commute with the differential operators $\GRAD$,
  $\ROT$, and $\DIV$.  We also construct mollification operators
  satisfying boundary conditions and use them to characterize the
  kernel of traces related to the tangential and normal trace of
  vector fields.  We use the mollification operators to build
  projection operators onto general $H^1$-, $\bH(\text{curl})$- and
  $\bH(\text{div})$-conforming finite element spaces, with and without
  homogeneous boundary conditions. These operators commute with the
  differential operators $\GRAD$, $\ROT$, and $\DIV$, are $L^p$-stable, and have optimal approximation
  properties on smooth functions.
\end{abstract}



\pagestyle{myheadings} \thispagestyle{plain} 
\markboth{A. ERN, J.L. GUERMOND}{Mollification in Lipschitz domains}

\section{Introduction}
Smoothing by mollification is an important tool for the analysis and
the approximation of partial differential equations. This tool has
been introduced by \citet[p.~206]{MR1555394}, \citet[p.~487]{Sob38},
and \citet[p.~136--139]{MR0009701}. The mollifying operation commutes
with differential operators and converges optimally when the function
to be smoothed is defined over the entire space $\Real^d$. These
properties may not be easy to achieve if the function in question is
defined only in a non-smooth domain $\Dom$, say only Lipschitz, since
mollification by convolution as originally introduced in the above
references requires to extend the function outside $\Dom$, which is
a non-trivial task in general unless the boundary and the function are
reasonably smooth.  The boundary difficulty has
been overcome by \citet{MR1884727} and \citet{MR1732050} by 
redefining mollification using a convolution-translation technique
so that mollification does not require information outside of
$\Dom$. The question is, however, more subtle when
dealing with vectors fields where normal or tangent boundary
conditions are involved, and, to our knowledge, has not yet been fully
addressed in the literature.  

The first objective of this paper is to revisit the theory of
mollification for scalar- and vector-valued fields in strongly
Lipschitz domains with the following goals in mind: the mollification
operators must be compatible with the De Rham complex (\ie they
must commute with the standard differential operators $\GRAD$, $\ROT$, and
$\DIV$), be $L^p$-stable for any real number $p\in [1,\infty]$, and
have strong convergence properties in the entire domain.  Using a
partition of unity subordinated to a covering of the boundary as done
in \citet{MR1884727} and \citet{MR1732050} is (seemingly) incompatible
with the first constraint.  The route that we propose instead is based
on a shrinking technique of the domain using transversal vector
fields.  We also devise a second sequence of mollification operators
based on the extension by zero; the corresponding operators have the
property that the mollified functions are in the kernel of the
canonical trace operators in $H^1$, $\Hrt$, and $\Hdv$ (\ie zero
trace, zero tangential trace, and zero normal trace, respectively).
The main results of the first part of the paper are
Theorems~\ref{Thm:Kdeltaf_to_f_in_LP} and
\ref{Thm:Kdeltaf_to_f_in_LP_zero}, and their respective corollaries.
As an application of independent interest, we give in
Theorem~\ref{Thm:traces_div_curl} a clear characterization of the
kernel of the traces associated with the divergence and the curl
operators.

The second objective of this work is to use the mollification
operators introduced in the first part of the paper to construct
quasi-interpolation operators onto general finite element subspaces of
$H^1$, $\Hrt$, and $\Hdv$, with and without homogeneous boundary
conditions. We want these operators to satisfy three key properties:
(i) $L^p$-stability for any $p\in[1,\infty]$, (ii) commutation with
differential operators, (iii) preservation of functions in the finite
element spaces.  Operators with such properties are important tools in
finite element exterior calculus, see
\citet[\S5.4]{Arnold_Falk_Winther_2006},
\cite[\S5.3]{Arnold_Falk_Winther_2010}, where they are termed bounded
cochain projections.  In particular, the above properties imply that
the quasi-interpolation error is bounded by the best approximation
error.  

The bases for constructing stable, commuting, and quasi-interpolation
projections have been laid out in \citet{Schoberl_2001,Schoberl_2008}
and~\citet{Christ:07}, where stability and commutation are achieved by
composing the canonical finite element interpolation operators with
some mollification technique. Then, following~\citet{Schoberl:05}, the
projection property over finite element spaces is obtained by
composing these operators with the inverse of their restriction to the
said spaces.  An important extension of this construction allowing the
possibility of using shape-regular mesh sequences and boundary
conditions has been achieved
by~\citet{Christiansen_Winther_mathcomp_2008}. Further variants of
this construction have recently been proposed. For instance
in~\citet{Christiansen:15}, the bounded cochain has the additional
property of preserving polynomials locally, up to a certain degree,
and in~\citet{FalWi:12} it is defined locally.  In the present work, we
revisit the results of~\cite{Christiansen_Winther_mathcomp_2008} by
invoking our shrinking-based mollification operators which do not
require extension outside the domain. We also present the results in
the language of numerical analysis to make them accessible to a wide
audience.  The main result of this second part is
Theorem~\ref{Th:Pih}. As an application, we give in
Theorem~\ref{Thm:Poincare} discrete Poincar\'e inequalities for
vector-valued functions.

The paper is organized as follows. In \S\ref{Sec:shrinking} we
introduce a shrinking technique of $\Dom$ that avoids difficulties with the
boundary. We also introduce an expansion technique.  
In~\S\ref{Sec:Mollification} we use the shrinking technique to devise
mollification operators that commute with differential operators, are
stable in any $L^p$ space, and have approximation properties on smooth
functions. Using the expansion technique from \S\ref{Sec:shrinking},
we introduce in \S\ref{Sec:Mollification_zero} mollification operators
that produce compactly supported functions and
share the same properties as the shrinking-based operators. 
In~\S\ref{sec:FE}, we
introduce the finite element setting that is necessary to construct
canonical interpolation operators on standard $H^1$-,
$\bH(\text{curl})$-, and $\bH(\text{div})$-conforming finite element
spaces, with and without homogeneous boundary
conditions. In~\S\ref{sec:quasi.inter} we devise quasi-interpolant
operators that are $L^p$-stable, commute with
differential operators, and preserve finite element spaces, with and without boundary conditions. 

\section{Some geometry} \label{Sec:shrinking} In the entire
paper, $\Dom$ is an open, bounded, strongly Lipschitz, connected set in $\R^d$,  $\interior(\Dom)$
denotes the interior of $\Dom$, and $\overline \Dom$ its closure.
Points in $\Real^d$ and $\Real^d$-valued functions and mappings are
denoted using bold face; the Euclidean norm in $\Real^q$, $q\ge1$, is denoted
$\|\SCAL\|_{\ell^2(\Real^q)}$, or $\|\SCAL\|_{\ell^2}$ when the
context is unambiguous. We abuse the notation by using the same symbol
for the $\ell^2$-induced matrix norm and the norm of multilinear maps.

\subsection{Shrinking of strongly Lipschitz domains}\label{Sec:mollification_Lipschitz}
Let $\Dom$
be a strongly Lipschitz domain, \ie there are $\alpha>0$,
$\beta>0$, a finite number $N$ of affine maps $\bT_n$,
$n\in\intset{1}{N}$, and Lipschitz maps
$\Phi_n:B(\bzero_{\Real^{d-1}},\alpha) \longrightarrow \Real$ such
that
$\front = \bigcup_{n=1}^N \bT_n(\bset (\bx,z=\Phi_n(\bx))\tq \bx\in
B(\bzero_{\Real^{d-1}},\alpha) \eset),$ and for all $n\in\intset{1}{N}$,
\begin{align*}
\bT_n(\bset  (\bx,z) \tq \bx\in
B(\bzero_{\Real^{d-1}},\alpha),\;\Phi_n(\bx)<z<\Phi_n(\bx)+\beta \eset) 
&\subset \interior(\Dom),\\
\bT_n(\bset  (\bx,z) \tq \bx\in B(\bzero_{\Real^{d-1}},\alpha),\;\Phi_n(\bx)-\beta<z<\Phi_n(\bx) \eset) 
&\subset \Real^d {\setminus}
\overline  \Dom,
\end{align*}
where $B(\bzero_{\Real^{d-1}},\alpha)$ is the open ball of radius
$\alpha$ in $\Real^{d-1}$ centered at the origin.

Since $\Dom$ is strongly Lipschitz and bounded, the combination of
Theorem~2.7 and Lemma~2.2 from \citet{Hofmann_Mitrea_Taylor_2007}
implies that $\Dom$ has continuous globally transversal vector fields,
\ie there exist a vector field $\bJ\in \bC^0(\front)$ and a real
number $\gamma>0$ with the property that
$\bn(\bx)\SCAL \bJ(\bx) \ge \gamma$ at \ae point $\bx$ on $\front$,
where $\bn$ is the unit normal vector pointing outward.
Proposition~2.3 in \citep{Hofmann_Mitrea_Taylor_2007} in turn implies
the existence of a vector field $\bj\in \bC^\infty(\Real^d)$ whose
restriction to $\front$ is
globally transversal and $\|\bj(\bx)\|_{\ell^2}=1$ for all
$\bx\in \front$. We then
define the mapping:
\begin{equation}
\bvarphi_\delta: \Real^d \ni \bx \mapsto \bx - \delta \bj(\bx)\in \Real^d.
\label{map_varphi_delta_Lipschitz}
\end{equation}
Using Proposition~4.15 in \cite{Hofmann_Mitrea_Taylor_2007}, together
with the uniform cone property (see
\citep[pp.~599-600]{Hofmann_Mitrea_Taylor_2007}), we infer that there
exists $\radius>0$ such that
\begin{equation}
\bvarphi_\delta(\Dom) + B(\bzero,\delta\radius) \subset
\Dom,\qquad \forall \delta\in[0,1].  \label{map_varphi_delta_compact_inclusion}
\end{equation}
%
\begin{lemma}[Properties of $\bvarphi_\delta$] \label{Lem:shrinking_Lipschitz} The following properties hold:
\berom
\item \label{Item0:smoothness_phi_delta} The map $\bvarphi_\delta$ is
  of class $C^\infty$ for all $\delta\in[0,1]$.
\item \label{Item3:smoothness_phi_delta} For all $\ell\in \Natural$, there is $c$ such that
  $\max_{\bx\in\Dom}\|D^\ell \bvarphi_\delta(\bx) - D^\ell \bx\|_{\ell^2} \le
  c\, \delta$, for all $\delta\in[0,1]$, where $D^\ell$ denotes the Fr\'echet derivative of order $\ell$.
\item \label{Item1:smoothness_phi_delta}
  $\bvarphi_\delta(\Dom)+B(\bzero,\delta\radius) \subset\Dom$ for
  all $\delta\in(0,1]$.
\item \label{Item2:smoothness_phi_delta}  The mapping
  $\bx\mapsto \bx + t (\bvarphi_\delta(\bx) + (\delta \radius \by)
  - \bx)$
  maps $\Dom$ into $\Dom$ for all $t\in [0,1]$, all
  $\by\in B(\bzero,1)$ and all $\delta \in[0,1]$.
\eerom
\end{lemma}

\begin{proof} The first two properties are consequences of $\bj$
  being of class $C^\infty$ and $\Dom$ being bounded, while~\eqref{Item1:smoothness_phi_delta} is just~\eqref{map_varphi_delta_compact_inclusion}. To prove~\eqref{Item2:smoothness_phi_delta}, observe that
  $t(\bvarphi_\delta(\bx)-\bx)=\bvarphi_{t\delta}(\bx)-\bx$.  This
  implies that
  $\bx + t (\bvarphi_\delta(\bx) + (\delta \radius \by) - \bx) =
  \bvarphi_{t\delta}(\bx) + t \delta \radius \by \in \bvarphi_{t
    \delta}(D) + B(\bzero,t\delta \radius)\subset\Dom$
  for all $\by\in B(\bzero,1)$, all $t\in [0,1]$ and all
  $\delta\in [0,1]$. 
\end{proof}

\subsection{Expansion of strongly Lipschitz domains}
Since $\Dom$ is bounded, there are $\bx_\Dom\in\Real^d$ and $r_\Dom>0$
such that $\Dom\subset B(\bx_\Dom,r_\Dom)$. Let
$\calO = B(\bx_\Dom,r_\Dom){\setminus}\overline{\Dom}$.  The domain
$\calO$ is bounded, open, and strongly Lipschitz; hence, we can apply
the above argument again, and deduce the existence of a vector field
$\bk\in \bC^\infty(\Real^d)$ that is globally transversal for $\calO$,
points outward $\Dom$,
and $\|\bk(\bx)\|_{\ell^2}=1$ for all $\bx\in \partial O$; note that
$\front \subset \partial O$. We then define the mapping:
\begin{equation}
\bvartheta_\delta: \Real^d \ni \bx \longmapsto \bx + \delta \bk(\bx)\in \Real^d.
\label{map_vartheta_delta_Lipschitz}
\end{equation}
As above, we infer that there
exists $\zeta>0$ such that
\begin{equation}
  \bvartheta_\delta(\calO) + B(\bzero,3\delta\zeta) \subset
  \calO,\qquad \forall \delta\in[0,1].  \label{map_vartheta_delta_compact_inclusion}
\end{equation}
\begin{lemma}[Properties of
  $\bvartheta_\delta$] \label{Lem:enlarging_Lipschitz} The following
  properties hold: \berom
\item \label{Item0:smoothness_vartheta_delta} The map $\bvartheta_\delta$ is
  of class $C^\infty$ for all $\delta\in[0,1]$.
\item \label{Item3:smoothness_vartheta_delta} For all
   $\ell\in \Natural$, there is $c$ such that
    $\max_{\bx\in\Dom}\|D^\ell \bvartheta_\delta(\bx) - D^\ell
   \bx\|_{\ell^2} \le
   c\, \delta$, for all $\delta\in[0,1]$.
\item \label{Item1:smoothness_vartheta_delta}
  $\bvartheta_\delta(\overline\calO)+B(\bzero,2\delta\zeta) \subset \calO$ for
  all $\delta\in(0,1]$.
\eerom
\end{lemma}

\begin{proof}
The only novelty is the proof of~(iii). 
Let $\bx\in \overline\calO$, then there exists $\bz_\delta \in \calO$
such that $(1+\delta L_\bk)\|\bz_\delta-\bx\|_{\ell^2} < \delta \zeta$, where $L_\bk$ denotes the Lipschitz constant of the field $\bk$ in $\overline\calO$.  We observe that
$\bvartheta_\delta(\bx) + B(\bzero,2\delta \zeta) = \bvartheta_\delta(\bz_\delta) +
(\bvartheta_\delta(\bx)-\bvartheta_\delta(\bz_\delta)) + B(\bzero,2\delta \zeta) \subset \bvartheta_\delta(\calO) +
B(\bzero,\delta \zeta) + B(\bzero,2\delta \zeta) = \bvartheta_\delta(\calO) + B(\bzero,3\delta\zeta) \subset \calO$ owing to~\eqref{map_vartheta_delta_compact_inclusion}. 
\end{proof}

\section{Mollification without extension} \label{Sec:Mollification} We
introduce in this section a mollification technique in strongly Lipschitz domains that does not
require to invoke non-trivial extensions and that commutes with
differential operators. The mapping $\bvarphi_\delta:\Dom \longrightarrow\Dom$ and
$\radius>0$ are defined in \eqref{map_varphi_delta_Lipschitz} and
\eqref{map_varphi_delta_compact_inclusion}.
In what follows, $\Jac_\delta(\bx)$ denotes the Jacobian matrix of $\bvarphi_\delta$ at $\bx\in \Dom$.

\subsection{Mollification} \label{Sec:Generic_assumptions} 
Let us consider the following kernel
\begin{equation}
\rho(\by):=\left\{\begin{array}{ll}%
    \eta\exp\left(-\frac 1{1-\|\by\|_{\ell^2}^2}\right), &\textnormal{if }\|\by\|_{\ell^2}< 1, \\
    0, &\textnormal{if }\|\by\|_{\ell^2} \ge 1,%
\end{array}\right.\label{def_mollification_kernel_rho}
\end{equation}
where $\eta$ is chosen so that
$\int_{\polR^d}\rho(\by)\dif y =\int_{B(\bzero,1)}\rho(\by)\dif y =
1$.
To be generic, we introduce $q\in \polN$, $q\ge1$, and a smooth
$\Real^{q\CROSS q}$-valued field
$\polA_\delta:\Dom\to \Real^{q\CROSS q}$ (related to the Jacobian $\Jac_\delta$ of $\bvarphi_\delta$, see~\eqref{eq:exp_polA} below) such that for all
$l\in \polN$, there is $c$ such that
\begin{equation}
\sup_{\bx \in \Dom}\|D^l(\polA_\delta(\bx) -\polI)\|_{\ell^2} \le c\, \delta,\label{Assumption_polA_delta}
\end{equation}  
where $\polI$ is
the identity matrix in $\Real^{q\CROSS q}$.
Consider the following smoothing operators acting on $f\in L^1(\Dom;\Real)$ and
$g=(g_1,\ldots,g_q)\tr\in L^1(\Dom;\Real^q)$:
\begin{subequations}\begin{align}
(\calK_\delta\upg f)(\bx) &:= \int_{B(\bzero,1)}
\rho(\by)  f(\bvarphi_\delta(\bx) + (\delta \radius) \by)\dif y,\label{def_calK_delta_upg}\\
(\calK_\delta g)(\bx) &:=  \polA_\delta(\bx) (\calK_\delta\upg g_1(\bx),\ldots,\calK_\delta\upg g_q(\bx))\tr.
\label{def_calK_delta}
\end{align}\end{subequations}
Property
\eqref{Item1:smoothness_phi_delta} from
Lemma~\ref{Lem:shrinking_Lipschitz} implies that
\[
\bvarphi_\delta(\bx) +
(\delta \radius) \by \in \bvarphi_{\delta}(\Dom) + (\delta
\radius)B(\bzero,1)\subset \Dom, \quad \forall \bx\in\Dom, \ \forall \by\in B(\bzero,1).
\] 
The means that the domains of $\calK_\delta\upg$ and $\calK_\delta$ are indeed
$L^1(\Dom;\Real)$ and $L^1(\Dom;\Real^q)$, \ie there is no need to
invoke extensions outside $\Dom$.
\begin{lemma}[Smoothness] \label{Lem:Kdelta_Cinfty}
$\calK_\delta g$ is in $C^\infty(\Dom;\Real^q)$ for all $g\in L^1(\Dom;\Real^q)$, and $\calK_\delta g$ as well as all its derivatives admit a continuous extension to $\overline\Dom$.
\end{lemma}
\begin{proof} 
 Owing to \eqref{Assumption_polA_delta} and \eqref{def_calK_delta}, and
 using the Leibniz product rule, 
  it suffices to show that the statement holds for $\calK_\delta\upg$. 
Let $f\in L^1(\Dom)$.
Let us prove first that $\calK_\delta\upg f$ is continuous.
Let $\bx$ and $\bz$ be two points in $\Dom$.
Up to  appropriate changes of variable we have
\begin{align*}
\calK_\delta\upg f(\bx) - \calK_\delta\upg f(\bz) & = \frac{1}{(\delta \radius)^{d}} \int_\Dom
\left(\rho\!\left(\tfrac{\by-\bvarphi_\delta(\bx)}{\delta\radius}\right)
-     \rho\!\left(\tfrac{\by-\bvarphi_\delta(\bz)}{\delta\radius}\right)\right)
f(\by)\dif y,
\end{align*}
where we replaced $\bvarphi_{\delta}(\bx) + (\delta \radius)B(\bzero,1)$ and
$\bvarphi_{\delta}(\bz) + (\delta \radius)B(\bzero,1)$ by $\Dom$ and used that
$\rho$ is zero outside the unit ball $B(\bzero,1)$.  The uniform Lipschitz
continuity of $\rho$ and $\bvarphi_\delta$ implies that there is $c$
such that $|\rho\left(\frac{\by-\bvarphi_\delta(\bx)}{\delta\radius}\right) -
\rho\left(\frac{\by-\bvarphi_\delta(\bz)}{\delta\radius}\right) |\le
\frac{c}{\delta \radius} \|\bx-\bz\|_{\ell^2}$. As a result, we infer that
\[
|\calK_\delta\upg f(\bx) - \calK_\delta\upg f(\bz)|
\le c (\delta \radius)^{-d-1} \|f\|_{L^1(\Dom)} \|\bx-\bz\|_{\ell^2},
\]
which proves that $\calK_\delta\upg f$ is Lipschitz continuous; hence
$\calK_\delta\upg f$ is uniformly continuous.  This proves that
$\calK_\delta\upg f\in C^0(\Dom;\Real)$ and $\calK_\delta\upg f$
admits a continuous extension to $\overline\Dom$. Let us now evaluate
the gradient of $\calK_\delta\upg f$. Using the chain rule, we infer
that
\begin{align*}
\GRAD (\calK_\delta\upg f)(\bx) &= \int_{B(\bzero,1)}
\rho(\by) \Jac_\delta\tr(\bx) (\GRAD f)(\bvarphi_\delta(\bx) + (\delta \radius) \by)\dif y\\
&=  \Jac_\delta\tr(\bx)  \int_{B(\bzero,1)}
\rho(\by)(\GRAD f)(\bvarphi_\delta(\bx) + (\delta \radius) \by)\dif y\\
&=  \Jac_\delta\tr(\bx) (\delta \radius)^{-1} \int_{B(\bzero,1)}
\rho(\by)\GRAD \big(f(\bvarphi_\delta(\bx) + (\delta \radius) \by)\big)\dif y\\
&=  - \Jac_\delta\tr(\bx) (\delta \radius)^{-1}  \int_{B(\bzero,1)}
\GRAD \rho(\by) f(\bvarphi_\delta(\bx) + (\delta \radius) \by)\dif y.
\end{align*}
We can then conclude that $\GRAD (\calK_\delta\upg f)$ is Lipschitz
continuous by using the same argument as above and continue the
argument by induction.
\end{proof}

\subsection{Examples}\label{Sec:Examples_without_BC}

Let $f\in L^1(\Dom;\Real)$ and $\bg\in
L^1(\Dom;\Real^d)$. 
Following \cite{Schoberl_2001} and \cite[\S3]{Schoberl_2008}, we define the following families of mollification operators:
\begin{subequations}\label{eq:calK}\begin{align}
(\calK_\delta\upg f)(\bx) &= \int_{B(\bzero,1)}
\rho(\by)  f(\bvarphi_\delta(\bx) + (\delta \radius) \by)\dif y, \label{def_calKg}\\
(\calK_\delta\upc \bg)(\bx) &= \int_{B(\bzero,1)}
\rho(\by)  \Jac_\delta\tr(\bx) \bg(\bvarphi_\delta(\bx) + (\delta \radius) \by)\dif y, \label{def_calKc}\\
(\calK_\delta\upd \bg)(\bx) &= \int_{B(\bzero,1)}
\rho(\by)  \det(\Jac_\delta(\bx)) \Jac_\delta^{-1}(\bx) \bg(\bvarphi_\delta(\bx) + (\delta \radius) \by)
\dif y, \label{def_calKd}\\
(\calK_\delta\upb f)(\bx) &= \int_{B(\bzero,1)}
\rho(\by) \det(\Jac_\delta(\bx)) f(\bvarphi_\delta(\bx) + (\delta \radius) \by)\dif y,\label{def_calKi}
\end{align}\end{subequations}
for all $\bx\in \Dom$. The superscripts in~\eqref{eq:calK} refer to
the fact that these operators are used to build projections onto
finite element spaces that are conforming in the graph space of the
gradient, curl, or divergence operator,  or onto a broken finite
element space (with no conformity requirement),
see Theorem~\ref{Th:Pih} below. The transformations involving $\Jac_\delta$ are related to the classical Piola transformations. Furthermore, the functions
$\calK_\delta\upc \bg$, $\calK_\delta\upd \bg$, $\calK_\delta\upb f$
are of the form \eqref{def_calK_delta} with $\polA_\delta\upg(\bx)= 1$
and
\begin{equation} \label{eq:exp_polA}
\polA_\delta\upc(\bx)=\Jac_\delta\tr(\bx),\quad
\polA_\delta\upd(\bx)=\det(\Jac_\delta(\bx)) \Jac_\delta^{-1}(\bx),\quad
\polA_\delta\upb(\bx)=\det(\Jac_\delta(\bx))\hspace{-10pt}
\end{equation}
Property (iii) from Lemma~\ref{Lem:shrinking_Lipschitz} implies that
\eqref{Assumption_polA_delta} holds true in the above three cases.
Let $p\in[1,\infty]$. Assuming $d=3$, we define
\begin{subequations}\begin{align}
Z\upgp(\Dom)&=\{f\in L^{p}(\Dom)\st \GRAD f\in \bL^{p}(\Dom)\},\label{eq:def_Wupgp}\\
\bZ\upcp(\Dom)&=\{\bv\in \bL^{p}(\Dom) \st \ROT \bv\in \bL^{p}(\Dom)\},\label{eq:def_Wupcp}\\
\bZ\updp(\Dom)&=\{\bv\in \bL^{p}(\Dom)\st \DIV \bv\in L^{p}(\Dom)\}. \label{eq:def_Wupdp}
\end{align}\end{subequations}
\begin{lemma}[Commuting with differential operators] \label{Lem:Kdelta_commutes}
The following holds:
\berom
\item \label{Item1:Lem:Kdelta_commutes}
$\GRAD \calK_\delta\upg f =\calK_\delta\upc \GRAD f$, for all $f\in Z\upgp(\Dom)$,
\item \label{Item2:Lem:Kdelta_commutes}
$\ROT \calK_\delta\upc \bg =\calK_\delta\upd \ROT \bg$, for all $\bg\in \bZ\upcp(\Dom)$,
\item \label{Item3:Lem:Kdelta_commutes}
$\DIV \calK_\delta\upd \bg =\calK_\delta\upb \DIV \bg$, for all $\bg\in \bZ\updp(\Dom)$,
\eerom
\ie the following diagram commutes:
\begin{equation}
\begin{diagram}[height=1.7\baselineskip,width=1.5cm]
Z\upgp(\Dom)  & \rTo^{\GRAD} & \bZ\upcp(\Dom) & \rTo^{\ROT} &  \bZ\updp(\Dom)
& \rTo^{\DIV} & L^p(\Dom) \\
\dTo_{\calK\upg_\delta} & & \dTo_{\calK\upc_\delta} & & \dTo_{\calK\upd_\delta}& & \dTo_{\calK\upb_\delta} \\
C^\infty(\Dom) & \rTo^{\GRAD} & \bC^\infty(\Dom) & \rTo^{\ROT} & \bC^\infty(\Dom) & \rTo^{\DIV} & C^\infty(\Dom) 
\end{diagram}
\label{Diag:W_C_Kdelta}
\end{equation}
\end{lemma}
\begin{proof}
Upon setting $\trans(\bx) = \bvarphi_\delta(\bx) + \delta \radius \by$,
these identities are simple consequences of the chain rule:
\begin{align*}
\GRAD(f{\circ}\trans)(\bx) &= \Jac_\delta\tr(\bx) (\GRAD f)(\trans(\bx)), \\
\ROT(\Jac_\delta\tr(\bx) (\bg{\circ}\trans))(\bx) &= \det(\Jac_\delta(\bx)) \Jac_K^{-1}(\bx) (\ROT \bg)(\trans(\bx)), \\
\DIV(\det(\Jac_\delta(\bx)) \Jac_\delta^{-1} (\bg{\circ}\trans))(\bx) &=
\det(\Jac_\delta(\bx)) (\DIV \bg)(\trans(\bx)).
\end{align*}
This completes the proof.
\end{proof}

\subsection{Convergence} \label{Sec:Convergence}
We now show that the smoothing operators defined above 
have interesting approximation properties.
Owing to Lemma~\ref{Lem:shrinking_Lipschitz},
$\Jac_\delta$ and $\Jac_\delta^{-1}$ converge uniformly to the
identity, and $\det(\Jac_\delta)$ converges uniformly to $1$. As a result,
there is $\delta_0\in (0,1]$ such that $\|\Jac_\delta-\polI\|_{\ell^2}\le \frac12$,  $\|\Jac_\delta^{-1}\|_{\ell^2}\le 2$, and $|\det(\Jac_\delta^{-1})|\le 2^d$, for all $\delta\in [0,\delta_0]$ and all $\bx\in\Dom$. 
 
\begin{theorem}[Convergence] \label{Thm:Kdeltaf_to_f_in_LP} The sequence
  $(\calK_\delta)_{\delta\in [0,\delta_0]}$ is uniformly bounded in
  $\calL(L^p;L^p):=\calL(L^p(\Dom;\Real^q);L^p(\Dom;\Real^q))$ 
for all $p\in [0,\infty]$. Moreover, for $p\in [1,\infty)$,
  $\|\calK_\delta f - f\|_{L^p(\Dom;\Real^q)}\to 0$ as $\delta \to 0$ for
  all $f\in L^p(\Dom;\Real^q)$. 
\end{theorem}
\begin{proof}
  Owing to \eqref{Assumption_polA_delta} and \eqref{def_calK_delta},
  it suffices to show that the statement holds for $\calK_\delta\upg$.
  (1) We show first that $\calK_\delta\upg$ is uniformly bounded in
  $\calL(L^p;L^p)$ by using the Riesz--Thorin interpolation theorem. The
  statement is evident for $p=\infty$ with constant $c=1$. 
  Now consider $f\in L^1(\Dom;\Real)$, then
\begin{align*}
\|\calK_\delta\upg f \|_{L^1(\Dom)} &\le \int_\Dom \int_{B(\bzero,1)} 
\rho(\by) |f(\bvarphi_\delta(\bx) + (\delta \radius)\by)|\dif y \dif x \\
& \le
\int_{B(\bzero,1)}  \rho(\by) \int_{\Dom}
 |f(\bvarphi_\delta(\bx) + (\delta \radius)\by)|\dif x \dif y \\
& \le \int_{B(\bzero,1)} \rho(\by) \int_{\bvarphi_\delta(\Dom) + (\delta \radius)\by} |f(\bz)| |\det(\Jac_\delta(\bz))|^{-1}\dif z \dif y
\le c \|f\|_{L^1(\Dom)},
\end{align*}
since $\delta\le\delta_0$. 
The Riesz--Thorin interpolation theorem implies that
$\|\calK_\delta\upg f \|_{L^p(\Dom)} \le c^{\frac1p}
\|f\|_{L^p(\Dom)}$,
so that $\|\calK_\delta\upg f \|_{L^p(\Dom)} \le c_1 \|f\|_{L^p(\Dom)}$, with $c_1=\max(1,c)$.\\
(2) Assume first that $f$ is smooth over $\Dom$, 
say uniformly Lipschitz with Lipschitz constant $L_f$, \ie $|f(\bx)-f(\bz)|\le L_f \|\bx-\bz\|_{\ell^2}$. Then, 
\begin{align*}
\left|\calK_\delta\upg f(\bx) - f(\bx)\right| & = \left|\int_{B(\bzero,1)} \rho(\by)
\big(f(\bvarphi_\delta(\bx) + (\delta\radius) \by) - f(\bx)\big)\dif y\right|\\ & \le
\int_{B(\bzero,1)} \rho(\by) L_f \|\bvarphi_\delta(\bx) - \bx + (\delta\radius) \by)
\|_{\ell^2}\dif y \le c\, L_f \delta.
\end{align*}
In conclusion, there is $c_0=c\max(1,|D|)$ such that 
$\|\calK_\delta\upg f - f\|_{L^p(\Dom)}\le c_0\, L_f \delta.$\\
(3) We conclude by using a density argument and the triangle
inequality. Let $f\in L^p(\Dom)$.  The space of uniformly Lipschitz
functions being dense in $L^p(\Dom)$, there is a sequence of uniformly Lipschitz
functions $(f_n)_{n\in \Natural}$ such that $\|f_n-f\|_{L^p(\Dom)}\to
0$ as $n\to \infty$. Then
\begin{align*}
\|\calK_\delta\upg f - f\|_{L^p(\Dom)} 
& \le \|\calK_\delta\upg (f - f_n)\|_{L^p(\Dom)} + \|\calK_\delta\upg f_n - f_n\|_{L^p(\Dom)} 
+ \|f_n - f\|_{L^p(\Dom)} \\
&\le c_1  \|f - f_n\|_{L^p(\Dom)} + c_0 L_{f_n} \delta + \|f_n - f\|_{L^p(\Dom)}. 
\end{align*}
Let $\epsilon>0$ and let $n(\epsilon)$ be large enough so that
$\|f_{n(\epsilon)} - f\|_{L^p(\Dom)}\le \epsilon$.  Setting
$\delta_0(\epsilon) = \epsilon/L_{f_{n(\epsilon)}}$, we have
$\|\calK_\delta\upg f - f\|_{L^p(\Dom)} \le \epsilon(c_1 + c_0 + 1) $
for all $\delta\le \delta_0(\epsilon)$.  In conclusion, for all
$\epsilon>0$, there is $\delta_0(\epsilon)$ such that
$\|\calK_\delta\upg f - f\|_{L^p(\Dom)} \le \epsilon(c_1 + c_0 + 1) $
for all $\delta\le \delta_0(\epsilon)$, which proves that
$\|\calK_\delta\upg f - f\|_{L^p(\Dom)} \to 0$ as $\delta \to 0$.
\end{proof}

We now can state a result that shows that the above smoothing
technique is superior to mollification alone, \ie contrary to the
result originally stated by Friedrichs (see \eg
\cite[Thm~9.2]{Brezis:11}), strong convergence on the derivatives now
occurs over the entire domain $\Dom$.
\begin{corollary}[Convergence of derivatives] \label{Cor:Kdeltaf_to_f_in_LP}
Let $p\in [1,\infty)$. Then,
\begin{subequations}\begin{align}
\lim_{\delta\to 0}\|\GRAD(\calK\upg_\delta f - f)\|_{\bL^{p}(\Dom)} = 0, 
& \quad \forall f\in  Z\upgp(\Dom),\label{Grad_convergence}\\
\lim_{\delta\to 0}\|\ROT(\calK\upc_\delta \bg - \bg)\|_{\bL^{p}(\Dom)} = 0, 
& \quad \forall \bg\in  \bZ\upcp(\Dom),\label{Rot_convergence}\\
\lim_{\delta\to 0}\|\DIV(\calK\upd_\delta \bg - \bg)\|_{L^{p}(\Dom)} = 0, 
& \quad \forall \bg\in \bZ\updp(\Dom). \label{Div_convergence}
\end{align}\end{subequations}
\end{corollary}
\begin{proof}
  Using Lemma~\ref{Lem:Kdelta_commutes} we infer that
  $\GRAD \calK_\delta\upg f = \calK_\delta\upc \GRAD f$, and
  Theorem~\ref{Thm:Kdeltaf_to_f_in_LP} implies that
  $\calK_\delta\upc \GRAD f \to \GRAD f$ in $\bL^p(\Dom)$ as
  $\delta\to 0$, which proves \eqref{Grad_convergence}.  A similar
  argument holds for \eqref{Rot_convergence} and
  \eqref{Div_convergence}.
\end{proof}

\subsection{Convergence rate}
We now establish convergence rates.
\begin{theorem}[Convergence rate] \label{Th:approx_Kdelta_W1p}
 There is $c$ such that
\begin{equation}
  \|\calK_\delta f -f\|_{L^p(\Dom;\Real^q)}\le c\,\delta^s 
  |f|_{W^{s,p}(\Dom;\Real^q)},
\end{equation} 
for all $f\in W^{s,p}(\Dom;\Real^q)$, all $\delta\in
[0,\delta_0]$, and all $s\in (0,1]$, $p\in [1,\infty)$ or $s=1$, $p\in [1,\infty]$.
\end{theorem}
\begin{proof}Owing to \eqref{Assumption_polA_delta} and \eqref{def_calK_delta},
  it suffices to show that the statement holds for $\calK_\delta\upg$. Assume first that $p<\infty$.\\
(1)  Let $f\in W^{s,p}(\Dom)$ with $s\in (0,1)$. We estimate $\calK_\delta\upg f - f$ in
  $L^p(\Dom)$ as follows: 
\begin{multline*}
\|\calK_\delta\upg f - f\|_{L^p(\Dom)}^p  =
\int_\Dom\bigg|\int_{B(\bzero,1)} \rho(y)\big(f(\bvarphi_\delta(\bx)
+(\delta\radius)\by) - f(\bx)\big)\dif y\bigg|^p \dif x \\ 
 \le c
\int_{B(\bzero,1)} \int_\Dom \frac{\left|f(\bvarphi_\delta(\bx) +(\delta\radius)\by) -
f(\bx)\right|^p}{\|\bvarphi_\delta(\bx) +(\delta\radius)\by -\bx\|_{\ell^2}^{sp+d}} 
\|\bvarphi_\delta(\bx) +(\delta\radius)\by -\bx\|_{\ell^2}^{sp+d} \dif x \dif y.
\end{multline*}
Let us make the change of variables
$B(\bzero,1)\ni \by \mapsto \bz =\bvarphi_\delta(\bx) +
(\delta\radius)\by \in \bvarphi_\delta(\Dom) + \delta\radius B(\bzero,1)
\subset \Dom$.
Observe that the Jacobian of this transformation is bounded from above
by $\delta\radius$ and
\[
\|\bvarphi_\delta(\bx) +
(\delta\radius)\by  - \bx\|_{\ell^2} \le \|\bvarphi_\delta(\bx) -\bx\|_{\ell^2} 
+ \delta\radius\|\by\|_{\ell^2}\le c\, \delta.
\]
Hence,
\begin{align*}
\|\calK_\delta\upg f - f\|_{L^p(\Dom)}^p  \le c\, \delta^{sp+d} \delta^{-d}
\int_{\Dom} \int_\Dom \frac{\left|f(\bz) -
f(\bx)\right|^p}{\|\bz-\bx\|_{\ell^2}^{sp+d}} \dif x \dif z
\le c\, \delta^{sp} |f|_{W^{s,p}(\Dom)}^p.
\end{align*}
(2) Let $f\in W^{1,p}(\Dom)$. By proceeding as above we infer that 
\begin{align*}
\|\calK_\delta\upg f - f\|_{L^p(\Dom)}^p & \le c
\int_{B(\bzero,1)} \int_\Dom \left|f(\bvarphi_\delta(\bx) +(\delta\radius)\by) -
f(\bx)\right|^p\dif x \dif y.
\end{align*}
Let us fix $y\in B(\bzero,1)$ and define the mapping $\bpsi_\delta:\Dom\ni
\bx\mapsto \bvarphi_\delta(\bx) + (\delta\radius)\by \in \bvarphi_\delta(\Dom) +
\delta\radius B(\bzero,1)\subset \Dom$. Observe that 
\[
\|\bpsi_\delta(\bx) - \bx\|_{\ell^2} \le \|\bvarphi_\delta(\bx) -\bx\|_{\ell^2} 
+ \delta\radius\|\by\|_{\ell^2}\le c\, \delta,\quad
\|D\bpsi_\delta(\bx) -\polI\| = \|D\bvarphi_\delta(\bx) -\polI\| \le c\, \delta,
\]
and
$\bx+t(\bpsi_\delta(\bx) -\bx) = \bx+t(\bvarphi_\delta(\bx)+\delta\radius \by
-\bx) \in \Dom$,
\ie $\bpsi_\delta$ satisfies the assumptions of
Lemma~\ref{Lem:fpsilambda_minus_f_in_w1p} below. Hence,
$\int_\Dom \left|f(\bvarphi_\delta(\bx) +(\delta\radius)\by) -
  f(\bx)\right|^p\dif x\le c\, \delta^p \|\GRAD f\|_{\bL^p(\Dom)}^p$.
We conclude that $\|\calK_\delta\upg f - f\|_{L^p(\Dom)} \le c \,
\delta \|\GRAD f\|_{\bL^p(\Dom)}$. \\
(3) The case $s=1$, $p=\infty$ is treated similarly to (2).
\end{proof}

\begin{corollary}[Convergence rate on derivatives] 
Let $s\in (0,1)$, $p\in [1,\infty)$ or $s=1$, $p\in [1,\infty]$. Then,
\label{Cor:Kdeltaf_to_f_in_LP_rate} there is $c$ such that 
\begin{alignat*}{2}
\|\GRAD(\calK\upg_{\delta} f - f)\|_{\bL^{p}(\Dom)} &\le c \, \delta^s |\GRAD f|_{\bW^{s,p}(\Dom)}, 
& \quad &\forall f\in  \bset v \in L^p(\Dom)\st \GRAD v\in \bW^{s,p}(\Dom)\eset, \\
\|\ROT(\calK\upc_\delta \bg - \bg)\|_{\bL^{p}(\Dom)} &\le c\, \delta^s |\ROT \bg|_{\bW^{s,p}(\Dom)},
& \quad &\forall \bg\in  \bset \bv\in \bL^p(\Dom) \st \ROT \bv \in  \bW^{s,p}(\Dom)\eset,\\
\|\DIV(\calK\upd_\delta \bg - \bg)\|_{L^{p}(\Dom)} &\le c\, \delta^s |\DIV \bg|_{W^{s,p}(\Dom)},
& \quad &\forall \bg\in \bset \bv \in \bL^p(\Dom)\st \DIV\bv \in W^{s,p}(\Dom)\eset.
\end{alignat*}
\end{corollary}

\begin{proof}  Let $f\in \bset v \in L^p(\Dom)\st \GRAD v\in \bW^{s,p}(\Dom)\eset$, then
\begin{align*}
\|\GRAD (\calK_{\delta}\upg f -f)\|_{\bL^p(\Dom)} &=
\|\calK_{\delta}\upc \GRAD f -\GRAD f\|_{\bL^p(\Dom)} 
&& \text{since $\GRAD\calK_{\delta}\upg = \calK_{\delta}\upc\GRAD$} \\
& \le c\, \delta^s \|\GRAD f\|_{\bW^{s,p}(\Dom)} &&\text{owing to Theorem~\ref{Th:approx_Kdelta_W1p}}.
\end{align*}
Proceed similarly for the two other estimates.
\end{proof}

\begin{lemma}[Approximation] \label{Lem:fpsilambda_minus_f_in_w1p} Let $\lambda_0>0$,
  and assume that
  $\bpsi_\lambda:\Dom\to
  \Dom$
  is a diffeomorphism of class $C^1$ such that
  $\|\bpsi_\lambda(\bx)-\bx\|_{\ell^2}\le c'\, \lambda$ and
  $\|D\bpsi_\lambda(\bx)-\polI\|_{\ell^2} \le \frac12$ for all $\bx\in \Dom$ and all
  $\lambda\in [0,\lambda_0]$.  Assume also that the mapping
  $\bmu_{\lambda,t}: \bx\mapsto \bx+t(\bpsi_\lambda(\bx)-\bx)$
  maps $\Dom$ into $\Dom$ for all $t\in [0,1]$ and all
  $\lambda\in [0,\lambda_0]$. Then, there is $c$ such that the
  following holds:
\[
\|f\circ\bpsi_\lambda -f\|_{L^p(\Dom)} \le c\, \lambda \|\GRAD f\|_{\bL^p(\Dom)},
\]
for all $\lambda\in [0,\lambda_0]$, all $f\in W^{1,p}(\Dom)$, and all $p\in [1, \infty]$. 
\end{lemma} 
\begin{proof} 
(1) Assume first that $f$ is smooth. Let $\bx\in\Dom$ and $v(t):=
  f(\bmu_{\lambda,t}(\bx))$ with $t\in [0,1]$.  The chain rule implies that
  $v'(t) = Df(\bmu_{\lambda,t}(\bx))(\bpsi_\lambda(\bx)-\bx)$, thereby showing that
\[
f(\bpsi_\lambda(\bx)) -f(\bx)=\int_0^1 v'(t)\dif t = \int_0^1
Df(\bmu_{\lambda,t}(\bx)) (\bpsi_\lambda(\bx)-\bx) \dif t.
\]
Then, assuming that $p<\infty$, we infer that
\begin{align*}
\|f{\circ}\bpsi_\lambda -f\|_{L^p(\Dom)}^p &\le \int_\Dom
\|\bpsi_\lambda(\bx)-\bx\|_{\ell^2}^p \int_0^1 \|\GRAD
f(\bmu_{\lambda,t}(\bx))\|_{\ell^2}^p\dif t\dif x\\
& \le c'\, \lambda^p \int_0^1 \int_\Dom
\|\GRAD f(\bmu_{\lambda,t}(\bx))\|_{\ell^2}^p\dif t\dif x.
\end{align*}
The assumptions on $\bpsi_\lambda$ imply that the map $\bmu_{\lambda,t}$ is
invertible and $\|D\bmu_{\lambda,t}^{-1}\|_{\ell^2}\le 2$,
$|\det(D\bmu_{\lambda,t}^{-1})|\le 2^d$. As a result,
\begin{align*}
\|f\circ\bpsi_\lambda -f\|_{L^p(\Dom)}^p &\le c'\, \lambda^p \int_0^1
\int_{\Dom} \|\GRAD f(\bz)\|_{\ell^2}^p |\det(D\bmu_{\lambda,t}^{-1})| \dif
z\dif t,
\end{align*}
which finally implies that there is $c_0$ so that $\|f\circ \bpsi_\lambda -f\|_{L^p(\Dom)} 
\le c_0\, \lambda\|\GRAD f\|_{\bL^p(\Dom)}.$
The case $p=\infty$ is treated similarly.
\\
(2) If $f$ is not smooth, we deduce from
Corollary~\ref{Cor:Kdeltaf_to_f_in_LP} that there exists a sequence of
smooth functions converging to $f$ in $W^{1,p}(\Dom)$, \ie for all
$\epsilon>0$, there is a smooth function $f_\epsilon$ such that
$\|f-f_\epsilon\|_{L^p(\Dom)}\le \epsilon$ and
$\|\GRAD f_\epsilon\|_{\bL^p(\Dom)} \le 2 \|\GRAD f\|_{\bL^p(\Dom)}$.
Then
\begin{align*}
\|f\circ\bpsi_\lambda - f\|_{L^p(\Dom)}& \le \|(f-f_\epsilon)\circ\bpsi_\lambda\|_{L^p(\Dom)}
+ \|f_\epsilon\circ\bpsi_\lambda - f_\epsilon\|_{L^p(\Dom)} + \|f_\epsilon -f\|_{L^p(\Dom)} \\
& \le  c \epsilon + 2 c_0 \lambda \|\GRAD f\|_{\bL^p(\Dom)}+ \epsilon.
\end{align*}
The conclusion follows readily since $\epsilon$ is arbitrary.
\end{proof}

\section{Mollification with extension by zero} \label{Sec:Mollification_zero}
Note that the function $\calK_\delta f$ defined in
\eqref{def_calK_delta} does not satisfy any particular boundary
condition.  For instance, even if $f$ is zero on $\front$,
$(\calK_\delta f)_{|\front}$ is not necessarily zero.  Since preserving
boundary conditions may be useful in some applications, we now
construct a mollifier that has this
property. Let $C_0^\infty(\Dom;\Real^q)$ denote the space of
$\Real^q$-valued functions that are of class $C^\infty$ and of compact
support in $\Dom$. Consider the mapping $\bvartheta_\delta$ and the constant $\zeta$
defined in \eqref{map_vartheta_delta_Lipschitz} and
\eqref{map_vartheta_delta_compact_inclusion}.
Let $\polK_\delta(\bx)$ denote
the Jacobian matrix of $\bvartheta_\delta$ at $\bx\in \Dom$.

\subsection{Mollification}
For any $g\in L^1(\Dom;\Real^q)$,  $q\in\Natural$ with $q\ge1$, we denote by $\tg$ the extension by
zero of $g$ over $\Real^d$, \ie $\tg(\bx)=g(\bx)$ if $\bx\in \Dom$ and
$\tg(\bx)=0$ otherwise.  Taking inspiration from
\citet{Bonito_guermond_luddens_2015}, we introduce
\begin{align}
(\calK_{\delta,0}\upg f)(\bx) &:= \int_{B(\bzero,1)}
\rho(\by)  \tf(\bvartheta_\delta(\bx) + (\delta \zeta) \by)\dif y,\label{def_calK_delta_upg_zero}\\
(\calK_{\delta,0} g)(\bx) &:=  \polB_\delta(\bx) (\calK_{\delta,0}\upg g_1(\bx),\ldots,\calK_{\delta,0}\upg g_q(\bx))\tr,
\label{def_calK_delta_zero}
\end{align}
for all $\bx\in \Dom$, all $f\in L^1(\Dom;\Real)$, and all
$g=(g_1,\ldots,g_q)\tr \in L^1(\Dom;\Real^q)$,
where $\polB_\delta$ is a smooth
$\Real^{q\CROSS q}$-valued field (related to the Jacobian $\polK_\delta$ of $\bvartheta_\delta$) such that for all $l\in \polN$, there is $c$ such that
\begin{equation}\label{eq:bound_polB}
\sup_{\bx \in \Dom}\|D^l(\polB_\delta(\bx) -\polI)\|_{\ell^2} \le c\, \delta.
\end{equation}
\begin{lemma}[Smoothness and boundary
  condition] \label{Lem:Kdelta_Cinfty_zero} 
 $\calK_{\delta,0} (g)$ is in $C_0^\infty(\Dom;\Real^q)$ for all
  $g\in L^1(\Dom;\Real^q)$ and all $\delta\in (0,1]$.
\end{lemma}
\begin{proof}
  The smoothness has already been proved in
  Lemma~\ref{Lem:Kdelta_Cinfty}.  Let $\kappa$ be the Lipschitz
  constant of the field $\bk$ over
  $\overline{\Dom}$. Let $\epsilon_\delta = \delta \zeta/(1+\delta\kappa)$. Let $\bx\in \Dom$
    be such that $\text{dist}(\bx,\front)< \epsilon_\delta$. Then,
    there exists a point $\bz\in \front$ such that
    $\text{dist}(\bx,\bz)\le \epsilon_\delta $, \ie
\begin{align*}
  \bvartheta_\delta(\bx) + B(\bzero,\delta\zeta) 
  & = \bvartheta_\delta(\bz) + B(\bzero,\delta\zeta) + \bvartheta_\delta(\bx) - \bvartheta_\delta(\bz)\\
  & = \bvartheta_\delta(\bz) + B(\bzero,\delta\zeta) + \bx - \bz + \delta (\bk(\bx) - \bk(\bz)) \\
  & \subset  \bvartheta_\delta(\bz) + B(\bzero,\delta\zeta) + B(\bzero,\epsilon_\delta + \delta \kappa \epsilon_\delta) \\
  & = \bvartheta_\delta(\bz) + B(\bzero,\delta\zeta + (1+\delta \kappa)\epsilon_\delta) 
    = \bvartheta_\delta(\bz) + B(\bzero,2\delta\zeta) \\ 
  &\subset \bvartheta_\delta(\overline{\calO}) + B(\bzero,2\delta \zeta) \subset \calO,
\end{align*}
owing to Lemma~\ref{Lem:enlarging_Lipschitz}(iii).
This implies that
$\bvartheta_\delta(\bx) + (\delta \zeta) \by  \subset \calO$
for all $\by\in B(\bzero,1)$, so that $(\calK_{\delta,0}\upg (f))(\bx) =0$
since $f(\bvartheta_\delta(\bx) + (\delta \zeta) \by)=0$ for all
$\by\in B(\bzero,1)$.  Hence, the support of $\calK_{\delta,0}\upg$ is
compact in $\Dom$. The same conclusion applies to $\calK_{\delta,0}$.
\end{proof}

\subsection{Examples} Let us proceed as in
\S\ref{Sec:Examples_without_BC}.
Let $\polB_\delta\upc(\bx)=\polK_\delta\tr(\bx)$,
$\polB_\delta\upd(\bx)=\det(\polK_\delta(\bx)) \polK_\delta^{-1}(\bx)$, and
$\polB_\delta\upb(\bx)=\det(\polK_\delta(\bx))$. Lemma~\ref{Lem:enlarging_Lipschitz}(ii)  
implies that \eqref{eq:bound_polB} holds for these choices of $\polB_\delta$.
Let $f\in L^1(\Dom;\Real)$ and $\bg\in L^1(\Dom;\Real^d)$. We define the following families of mollification operators:
\begin{subequations}\begin{align}
(\calK_{\delta,0}\upg f)(\bx) &= \int_{B(\bzero,1)}
\rho(\by)  \tf(\bvartheta_\delta(\bx) + (\delta \zeta) \by)\dif y, \label{def_calKg_zero}\\
(\calK_{\delta,0}\upc \bg)(\bx) &= \int_{B(\bzero,1)}
\rho(\by)  \polK_\delta\tr(\bx) \btg(\bvartheta_\delta(\bx) + (\delta \zeta) \by)\dif y, \label{def_calKc_zero}\\
(\calK_{\delta,0}\upd \bg)(\bx) &= \int_{B(\bzero,1)}
\rho(\by)  \det(\polK_\delta(\bx)) \polK_\delta^{-1}(\bx) \btg(\bvartheta_\delta(\bx) + (\delta \zeta) \by)
\dif y, \label{def_calKd_zerà}\\
(\calK_{\delta,0}\upb f)(\bx) &= \int_{B(\bzero,1)}
\rho(\by) \det(\polK_\delta(\bx)) \tf(\bvartheta_\delta(\bx) + (\delta \zeta) \by)\dif y,\label{def_calKi_zero}
\end{align}\end{subequations}
for all $\bx\in \Dom$. 
Let $p\in[1,\infty]$. If $d=3$, we define
\begin{subequations}\begin{align}
\tZ\upgp(\Dom)&= \bset f\in L^p(\Dom)\st \GRAD \tf\in \bL^p(\Real^d)\eset, \\ 
\btZ\upcp(\Dom)&= \bset \bv\in \bL^p(\Dom)\st \ROT \btv\in \bL^p(\Real^d)\eset, \\
\btZ\updp(\Dom)&=\bset \bv\in \bL^p(\Dom)\st \DIV \btv\in L^p(\Real^d)\eset.
\end{align}\end{subequations}
\begin{lemma}[Commuting extension and derivatives] \label{Lem:tilde_commutes_with_diff_operators} The following holds:
\berom
\item $\GRAD \tf = \widetilde{\GRAD f}$, for all $f\in \tZ\upgp(\Dom)$,
\item $\ROT \btg = \widetilde{\ROT \bg}$, for all $\bg\in \btZ\upcp(\Dom)$,
\item $\DIV \btg = \widetilde{\DIV \bg}$,  for all $\bg\in \btZ\updp(\Dom)$.
\eerom
\end{lemma}
\begin{proof}
   Let $f\in \tZ\upgp(\Dom)$ and let $\bphi\in \bC_0^\infty(\Real^d)$ be
  a (vector-valued) smooth function compactly supported in
  $\interior(\Real^d{\setminus}\Dom)$. Then,
\begin{align*}
\int_{\Real^d} \bphi\SCAL \GRAD \tf \dif x = - \int_{\Real^d} \tf \DIV\bphi \dif x 
= - \int_{\Dom} f \DIV\bphi \dif x =0.
\end{align*}
Since $\bphi$ is arbitrary, this proves that $\GRAD\tf$ is zero in $\Real^d{\setminus}\Dom$.
Now let $\bphi\in \bC_0^\infty(\Dom)$, then
\begin{align*}
  -\int_\Dom \bphi\SCAL \GRAD \tf \dif x  = -\int_{\Real^d} \bphi\SCAL \GRAD \tf \dif x 
  = \int_{\Real^d} \tf \DIV\bphi \dif x 
  = \int_{\Dom} f \DIV\bphi \dif x = -\int_\Dom \bphi \SCAL \GRAD f \dif x.
\end{align*}
Since $\bphi$ is arbitrary, this proves that
$(\GRAD \tf)_{|\Dom} = \GRAD f$. We have thus proved that
$\GRAD\tf = \widetilde{\GRAD f}$.  The argument for the other two
equalities is identical.
\end{proof}

\begin{lemma}[Commuting with differential operators] \label{Lem:Kdelta_commutes_zero}
The following holds: 
\berom
\item \label{Item1:Lem:Kdelta_commutes_0}
 $\GRAD \calK_{\delta,0}\upg f =\calK_{\delta,0}\upc \GRAD f$, for all
  $f\in \tZ\upgp(\Dom)$,
\item \label{Item2:Lem:Kdelta_commutes_0}
$\ROT \calK_{\delta,0}\upc \bg =\calK_{\delta,0}\upd \ROT \bg$, for all
  $\bg\in \btZ\upcp(\Dom)$,
\item \label{Item3:Lem:Kdelta_commutes_0}
  $\DIV \calK_{\delta,0}\upd \bg =\calK_{\delta,0}\upb \DIV \bg$, for
  all $\bg\in \btZ\updp(\Dom)$,  
\eerom
  \ie the following diagram commutes:
\begin{equation}
\begin{diagram}[height=1.7\baselineskip,width=1.5cm]
\tZ\upgp(\Dom)  & \rTo^{\GRAD} & \btZ\upcp(\Dom) & \rTo^{\ROT} &  \btZ\updp(\Dom)
& \rTo^{\DIV} & L^p(\Dom) \\
\dTo_{\calK\upg_{\delta,0}} & & \dTo_{\calK\upc_{\delta,0}} & & \dTo_{\calK\upd_{\delta,0}}& & \dTo_{\calK\upb_{\delta,0}} \\
C_0^\infty(\Dom) & \rTo^{\GRAD} & \bC_0^\infty(\Dom) & \rTo^{\ROT} & \bC_0^\infty(\Dom) & \rTo^{\DIV} & C_0^\infty(\Dom)
\end{diagram}
\end{equation}
\end{lemma}
\begin{proof}
  The proof is almost the same as that of
  Lemma~\ref{Lem:Kdelta_commutes_zero}.  For instance, using the chain
  rule together with
  Lemma~\ref{Lem:tilde_commutes_with_diff_operators}, we obtain
\begin{align*}
\GRAD \calK\upg_{\delta,0} f(\bx) &= \int_{B(\bzero,1)}\rho(\by)
\polK_\delta(\bx)\tr \GRAD \tf (\bvartheta(\bx) + \delta\zeta \by)\dif y  \\
&=\int_{B(\bzero,1)}\rho(\by)
\polK_\delta(\bx)\tr \widetilde{\GRAD f} (\bvartheta(\bx) + \delta\zeta \by)\dif y 
= \calK\upc_{\delta,0} \GRAD f(\bx).
\end{align*}
Note here that it is critical that $\GRAD \tf =\widetilde{\GRAD f}$.
The argument for the other two equalities is identical.
\end{proof}

\subsection{Convergence} Similarly to \S\ref{Sec:Convergence}, we can
now state convergence results. Owing to Lemma~\ref{Lem:enlarging_Lipschitz}, $\polK_\delta$ and
$\polK_\delta^{-1}$ converge uniformly to the identity and
$\det(\polK_\delta)$ converges uniformly to $1$. As a result,
there is $\tilde\delta_0\in (0,1]$ such that $\|\polK_\delta-\polI\|_{\ell^2}\le \frac12$,  $\|\polK_\delta^{-1}\|_{\ell^2}\le 2$, and $|{\det(\polK_\delta^{-1})}|\le 2^d$, for all $\delta\in [0,\tilde\delta_0]$ and all $\bx\in\Dom$.
We combine the
counterparts of Theorem~\ref{Thm:Kdeltaf_to_f_in_LP} and
Corollary~\ref{Cor:Kdeltaf_to_f_in_LP} into one statement.
\begin{theorem}[Convergence] \label{Thm:Kdeltaf_to_f_in_LP_zero} The
  sequence $(\calK_{\delta,0})_{\delta\in [0,\tilde\delta_0]}$ is
  uniformly bounded in $\calL(L^p;L^p)$ for all $p\in [1,\infty]$.
  Moreover, for all $p\in [1,\infty)$
\begin{equation}
\lim_{\delta\to 0}\|\calK_{\delta,0} f - f\|_{L^{p}(\Dom;\Real^q)} = 0, 
\quad \forall f\in  L^p(\Dom;\Real^q),\label{convergence_zero}
\end{equation}
and
\begin{subequations}\begin{align}
\lim_{\delta\to 0}\|\GRAD(\calK\upg_{\delta,0} f - f)\|_{\bL^{p}(\Dom)} = 0, 
& \quad \forall f\in  \tZ\upgp(\Dom),\label{Grad_convergence_zero}\\
\lim_{\delta\to 0}\|\ROT(\calK\upc_{\delta,0} \bg - \bg)\|_{\bL^{p}(\Dom)} = 0, 
& \quad \forall \bg\in  \btZ\upcp(\Dom),\label{Rot_convergence_zero}\\
\lim_{\delta\to 0}\|\DIV(\calK\upd_{\delta,0} \bg - \bg)\|_{L^{p}(\Dom)} = 0, 
& \quad \forall \bg\in \btZ\updp(\Dom). \label{Div_convergence_zero}
\end{align}\end{subequations}
\end{theorem}
\begin{proof}
  The proof of \eqref{convergence_zero} is the same as that of
  Theorem~\ref{Thm:Kdeltaf_to_f_in_LP}.  See the proof of
  Corollary~\ref{Cor:Kdeltaf_to_f_in_LP} for the other three statements.
\end{proof}

Let $s\in (0,1]$, $p\in [1,\infty)$ or $s=1$, $p\in [1,\infty]$. Let us denote by $\tW^{s,p}(\Dom;\Real^q)$, 
the space composed of the functions in $W^{s,p}(\Dom;\Real^q)$ whose extension by zero is in 
$W^{s,p}(\Real^d;\Real^q)$. We set $|f|_{\tW^{s,p}(\Dom;\Real^q)} := |\tf|_{W^{s,p}(\Real^d;\Real^q)}$.
\begin{theorem}[Convergence rate] \label{Th:approx_Kdelta_W1p_zero}
 There is $c$ such that
\begin{equation}
  \|\calK_{\delta,0} f -f\|_{L^p(\Dom;\Real^q)}\le c\,\delta^s 
  |f|_{\tW^{s,p}(\Dom;\Real^q)},
\end{equation} 
for all $f\in \tW^{s,p}(\Dom;\Real^q)$, all $\delta\in
[0,\tilde\delta_0]$, and all $s\in (0,1]$, $p\in [1,\infty)$ or $s=1$, $p\in [1,\infty]$.
\end{theorem}
\begin{proof}
The proof is identical to that of Theorem~\ref{Th:approx_Kdelta_W1p}.
\end{proof}

To state a convergence result using norms on $\Dom$, we recall (see \eg \citet[Thm.~1.4.2.4, Cor.~1.4.4.5]{Gr85}) that
\begin{subequations}\begin{alignat}{2}
\tW^{s,p}(\Dom;\Real^q) &= W_0^{s,p}(\Dom;\Real^q)&\qquad&  \text{if $sp\ne 1$},\\
\tW^{s,p}(\Dom;\Real^q) &= W^{s,p}(\Dom;\Real^q)&\qquad&  \text{if $sp\in [0,1)$}. \label{norm_equiv_sp_lt_one}
\end{alignat}\end{subequations}
(Recall also that the constants in the above norm equivalences depend on $|sp-1|$.)

\begin{corollary}[Convergence rate on derivatives] Let $p\in [1,\infty)$ and $s\in(0,\frac1p)$. Then,
\label{Cor:Kdeltaf_to_f_in_LP_zero} there is $c$ (depending on $|sp-1|$) such that 
\begin{alignat*}{2}
\|\GRAD(\calK\upg_{\delta,0} f - f)\|_{\bL^{p}(\Dom)} &\le c \, \delta^s \|\GRAD f\|_{W^{s,p}(\Dom)}, 
& \quad &\forall f\in  \bset v \in L^p(\Dom)\st \GRAD v\in \bW^{s,p}(\Dom)\eset,\\
\|\ROT(\calK\upc_{\delta,0} \bg - \bg)\|_{\bL^{p}(\Dom)} &\le c\, \delta^s \|\ROT \bg\|_{W^{s,p}(\Dom)},
& \quad &\forall \bg\in  \bset \bv\in \bL^p(\Dom) \st \ROT \bv \in  \bW^{s,p}(\Dom)\eset,\\
\|\DIV(\calK\upd_{\delta,0} \bg - \bg)\|_{L^{p}(\Dom)} &\le c\, \delta^s \|\DIV \bg\|_{W^{s,p}(\Dom)},
& \quad &\forall \bg\in \bset \bv \in \bL^p(\Dom)\st \DIV\bv \in W^{s,p}(\Dom)\eset.
\end{alignat*}
\end{corollary}
\begin{proof} The proof relies on the commuting properties from Lemma~\ref{Lem:Kdelta_commutes_zero},
  Theorem~\ref{Th:approx_Kdelta_W1p_zero}, and \eqref{norm_equiv_sp_lt_one}.
For instance,
\begin{align*}
\|\GRAD (\calK_{\delta,0}\upg f -f)\|_{\bL^p(\Dom)} &=
\|\calK_{\delta,0}\upc \GRAD f -\GRAD f\|_{\bL^p(\Dom)} 
&& \text{since $\GRAD\calK_{\delta,0}\upg = \calK_{\delta,0}\upc\GRAD$} \\
& \le c\, \delta^s |\GRAD f|_{\btW^{s,p}(\Dom)} &&\text{owing to Theorem~\ref{Th:approx_Kdelta_W1p_zero}}\\
& \le c'_{s,p} \delta^s \|\GRAD f\|_{\bW^{s,p}(\Dom)} && \text{owing to \eqref{norm_equiv_sp_lt_one}},
\end{align*}
where $c'_{s,p}$ depends on $|sp-1|$.
Proceed similarly for the two other estimates.
\end{proof}

\begin{remark}
  The construction of $\calK\upc_{\delta,0}$ is similar in spirit to what has
  been proposed in \citet{Bonito_guermond_luddens_2015}.  The curl
  estimates in Theorem~\ref{Th:approx_Kdelta_W1p_zero} and Corollary
  \ref{Cor:Kdeltaf_to_f_in_LP_zero} are identical to those in
  \cite[Thm.~3.1]{Bonito_guermond_luddens_2015}.
\end{remark}

\begin{remark}[$sp>1$]
  Convergence rates on derivatives can also be derived for
  $sp>1$, namely
  $\|\GRAD(\calK\upg_{\delta,0} f - f)\|_{\bL^{p}(\Dom)} \le c \,
  \delta^s |\GRAD f|_{\bW^{s,p}(\Dom)}$
  for all $f\in L^p(\Dom)$ with $\GRAD f\in \bW^{s,p}_0(\Dom)$,
  $\|\ROT(\calK\upc_{\delta,0} \bg - \bg)\|_{\bL^{p}(\Dom)} \le c\,
  \delta^s |\ROT \bg|_{\bW^{s,p}(\Dom)}$
  for all $\bg\in \bL^p(\Dom)$ with $\ROT \bg \in \bW^{s,p}_0(\Dom)$,
  and
  $\|\DIV(\calK\upd_{\delta,0} \bg - \bg)\|_{L^{p}(\Dom)} \le c\,
  \delta^s |\DIV \bg|_{W^{s,p}(\Dom)}$
  for all $\bg\in \bL^p(\Dom)$ with $\DIV\bg \in W^{s,p}_0(\Dom)$,
  where $c$ depends on $|sp-1|$.
  Note that these estimates require boundary conditions on the
  derivatives.
\end{remark}

\subsection{Traces of vector fields}
In this section, we illustrate the use of the mollifying operator
$\calK_{\delta,0}$.  Let $p\in (1,\infty)$. Recall the spaces
$\bZ\upcp(\Dom)$ and $\bZ\updp(\Dom)$
from~\eqref{eq:def_Wupcp}-\eqref{eq:def_Wupdp}.  Since the trace
operator
$\gamma_0: W^{1,p'}(\Dom) \longrightarrow W^{\frac{1}{p},p'}(\front)$
is surjective (see \citet[p.~315]{Brezis:11},
\citet[Thm.~1.5.1.2\&1.5.1.6]{Gr85}, \citet[Thm.~3.38]{MacLean_2000}
(for $s\in(\frac12,\frac32), p=2$)), letting
$\langle\SCAL,\SCAL\rangle_\front$ denote the duality pairing between
$\bW^{-\frac{1}{p},p}(\front)$ and $\bW^{\frac{1}{p},p'}(\front)$, we
define the bounded linear map
$\gamma_{\CROSS\bn}:\bZ\upcp(\Dom) \rightarrow
\bW^{-\frac{1}{p},p}(\front)$ by
\begin{align}
\langle \gamma_{\CROSS\bn}(\bv), \bl\rangle_\front
  := \int_\Dom \bv\SCAL \ROT \bw(\bl)\dif x - \int_\Dom \bw(\bl)\SCAL \ROT\bv\dif x 
\label{eq:int_by_parts_curl_curl},
\end{align}
for all $\bv\in \bZ\upcp(\Dom)$ and all
$\bl\in \bW^{\frac1p,p'}(\front)$, where
$\bw(\bl)\in \bW^{1,p'}(\Dom)$ is such that
$\gamma_0(\bw(\bl))=\bl$. Note that $\gamma_{\CROSS\bn}(\bv)=\bv_{|\front}\CROSS\bn$ when $\bv$ is smooth. The definition~\eqref{eq:int_by_parts_curl_curl} 
is independent of the choice
of $\bw(\bl)$. Indeed, let $\bw_1,\bw_2\in \bW^{1,p'}(\Dom)$ be such
that $\gamma_0(\bw_1)=\gamma_0(\bw_2)=\bl$, \ie
$\bw_1-\bw_2\in \bW^{1,p'}_0(\Dom)$.  Let $(\bphi_n)_{n\in \Natural}$
be a sequence in $\bC_0^\infty(\Dom)$ converging to $\bw_1-\bw_2$ in
$\bW^{1,p'}_0(\Dom)$. Then,
$0=\int_\Dom \bv\SCAL \ROT\bphi_n\dif x-\int_\Dom \bphi_n\SCAL \ROT
\bv\dif x $, as can be seen by replacing $\bv$ by $\calK_\delta\upc\bv$ and passing to the limit $\delta\to0$.
Passing to the limit $n\to\infty$ yields
$0=\int_\Dom \bv\SCAL \ROT(\bw_1-\bw_2)\dif x -\int_\Dom
(\bw_1-\bw_2)\SCAL \ROT\bv\dif x $;
hence,
$\langle \gamma_{\CROSS\bn}(\bv), \gamma_0(\bw_1)\rangle_\front
=\langle \gamma_{\CROSS\bn}(\bv), \gamma_0(\bw_2)\rangle_\front$,
which establishes the claim. 

We also define $\gamma_{\SCAL\bn}:\bZ\updp(\Dom) \rightarrow
W^{-\frac{1}{p},p}(\front)$ by
\begin{align}
\langle \gamma_{\SCAL\bn}(\bv), l\rangle_\front
  := \int_\Dom \bv\SCAL \GRAD q(l)\dif x + \int_\Dom q(l) \DIV \bv\dif x, 
\label{eq:int_by_parts_div_grad} 
\end{align}
for all $\bv\in \bZ\updp(\Dom)$ and all $l\in W^{\frac1p,p'}(\front)$, where
$q(l)\in W^{1,p'}(\Dom)$ is such that $\gamma_0(q(l))=l$, and
$\langle\SCAL,\SCAL\rangle_\front$ now denotes the duality pairing between
$W^{-\frac{1}{p},p}(\front)$ and $W^{\frac{1}{p},p'}(\front)$.  Reasoning as above, one can verify that this
definition is independent of the choice of $q(l)$. Note also that
$\gamma_{\SCAL\bn}(\bv) = \bv_{|\front}\SCAL \bn$ when $\bv$ is smooth.
%

Let us now introduce
\begin{subequations}\begin{align}
\bZ\upcp_0(\Dom)&:= \overline{\bC^\infty_0(\Dom)}^{\bZ\upcp(\Dom)},\\
\bZ\updp_0(\Dom)&:= \overline{\bC^\infty_0(\Dom)}^{\bZ\updp(\Dom)}.
\end{align}\end{subequations}
\begin{theorem}[Kernels of $\gamma_{\CROSS\bn}$ and
  $\gamma_{\SCAL\bn}$] \label{Thm:traces_div_curl} Let
  $p\in (1,\infty)$. Then,
\begin{subequations}
\begin{align}
\bZ\upcp_0(\Dom)&= \ker(\gamma_{\CROSS\bn}), \\
\bZ\updp_0(\Dom)&= \ker(\gamma_{\SCAL\bn}).
\end{align}\end{subequations}\vspace{-\baselineskip}
\end{theorem}
\begin{proof}
  Let us do the proof for $\gamma_{\CROSS\bn}$, the proof for
  $\gamma_{\SCAL\bn}$ is similar. 

  (1) We first show that
  $\bZ\upcp_0(\Dom)\subset \ker(\gamma_{\CROSS\bn})$, which is the
  easiest to establish.  By definition there is a sequence of smooth
  functions $(\bv_n)_{n\in\Natural}$ in $\bC^\infty_0(\Dom)$ converging to $\bv$ in
  $\bZ\upcp(\Dom)$. Let $\bw$ be a function in $\bC^\infty(\Dom)\cap
  \bC^0(\overline D)$,
  then
\[
0=\int_\Dom \DIV (\bw\CROSS \bv_n)\dif x 
= \int_\Dom \bv_n\SCAL \ROT\bw\dif x - \int_\Dom\bw\SCAL\ROT \bv_n\dif x.
\]
Both integrals on the right-hand side converge; hence,
\[
\langle\gamma_{\CROSS\bn}(\bv),\gamma_0(\bw) \rangle_\front 
= \int_\Dom\bv\SCAL \ROT\bw\dif x - \int_\Dom\bw\SCAL\ROT\bv\dif x = 0,
\]
for every function $\bw$ in $\bC^\infty(\Dom)\cap \bC^0(\overline D)$.
This also implies that the equality holds for all
$\bw\in W^{1,p'}(\Dom)$, since
$\bC^\infty(\Dom)\cap\bC^0(\overline D)$ is dense in
$\bw\in W^{1,p'}(\Dom)$, see Lemma~\ref{Lem:Kdelta_Cinfty} and
Theorem~\ref{Thm:Kdeltaf_to_f_in_LP}.  In conclusion,
$\bv\in \KER(\gamma_{\CROSS\bn})$ since $\gamma_0$ is surjective.

(2) Let us now establish the converse, \ie $\ker(\gamma_{\CROSS\bn})\subset \bZ\upcp_0(\Dom)$.
Let $\bv\in \ker(\gamma_{\CROSS\bn})$. Since $\bv\in \bZ\upcp(\Dom)\subset\bL^1(\Dom)$,
$\btv$ is differentiable in the distribution sense. 
Let $\bphi\in \bC_0^\infty(\Real^d)$, then
\begin{align*}
\langle \ROT\btv,\bphi\rangle=\int_{\Real^d} \btv\SCAL \ROT \bphi\dif x 
= \int_{\Dom} \bv\SCAL \ROT \bphi\dif x. 
\end{align*}
Using that $\bv\in \ker(\gamma_{\CROSS\bn})$, the above equality implies that 
\[
\langle \ROT\btv,\bphi\rangle = \int_{\Dom} \bv\SCAL \ROT \bphi\dif x
=\int_{\Dom} \bphi\SCAL\ROT \bv \dif x = \int_{\Real^d}
\bphi\SCAL\widetilde{\ROT\bv} \dif x.
\]
This proves that $\ROT\btv = \widetilde{\ROT\bv}\in \bL^1(\Real^d)$.
Hence $\bv\in \btZ\upcp(\Dom)$. We can now apply
\eqref{Rot_convergence_zero} from
Theorem~\ref{Thm:Kdeltaf_to_f_in_LP_zero} since
$\bv\in \btZ\upcp(\Dom)$, \ie the sequence
$(\calK_{\delta,0}\upc \bv)_{\delta\in [0,\tilde\delta_0]}$ converges to $\bv$ in
$\bZ\upcp(\Dom)$. This proves that
$\ker(\gamma_{\CROSS\bn})\subset \bZ\upcp_0(\Dom)$ since $\calK_{\delta,0}\upc\bv\in \bC^\infty_0(\Dom)$ (see Lemma~\ref{Lem:Kdelta_Cinfty_zero}).
\end{proof}

\begin{remark}[$Z_0=\tZ$]
The proof of Theorem~\ref{Thm:traces_div_curl} shows that $\bZ\upcp_0(\Dom) = \btZ\upcp(\Dom)$ and $\bZ\updp_0(\Dom) = \btZ\updp(\Dom)$; similarly, $Z\upgp_0(\Dom) = \tZ\upgp(\Dom)$.
\end{remark}

\section{Finite element setting}
\label{sec:FE}
We introduce in this section the finite element setting that we are
going to use in the rest of the paper.  We henceforth assume that
$\Dom$ is a bounded polyhedron in $\Real^d$.

\subsection{Meshes} 
Let $\famTh$ be a shape-regular sequence of affine meshes.  To avoid
technical questions regarding hanging nodes, we also assume that the
meshes cover $\Dom$ exactly and that they are matching, \ie for all
cells $K,K'\in\calT_h$ such that $K\ne K'$ and $K\cap K'\ne\emptyset$,
the set $K\cap K'$ is a common vertex, edge, or face of both $K$ and
$K'$ (with obvious extensions in higher space dimensions).  Given a
mesh $\calT_h$, the elements in $K\in \calT_h$ are closed sets in
$\Real^d$ by convention. The following sets
\begin{subequations}\begin{align}
  \calT_K &:= \bset K'\in \calT_h\st K'\cap K\ne \emptyset\eset,\\
\Dom_K &:= \interior \{\bx \in
\overline\Dom\tq \exists K'\in \calT_K,\, \bx \in K'\},
\end{align}\end{subequations}
for all $K\in\calT_h$, 
will be invoked in the following sections.
The set $\calT_K$ is the union of all the cells that touch $K$, and $\Dom_K$
is the interior of the collection of the points composing the cells in $\calT_K$.

We assume that there is a reference element $\wK$ such that for any
mesh $\calT_h$ and any cell $K\in \calT_h$, there is a bijective
affine mapping $\trans_K :\wK \longrightarrow K$ and an invertible
matrix $\Jac_K\in\Real^{d\CROSS d}$ (not to be confused with $\Jac_\delta$) such that
\begin{equation}
\trans_K(\wbx)-\trans_K(\wby) = \Jac_K (\wbx-\wby), \qquad \forall \wbx,\wby\in \wK.
\label{Eq2:TransAff}
\end{equation}
The shape-regularity assumption of the mesh sequence implies that
there are uniform constants $c^\sharp$, $c^\flat$ such
that
\begin{equation}
| {\det(\Jac_K)} | = \mes{K}\mes{\wK}^{-1}, \qquad 
\| \Jac_K \|_{\ell^2} \leq c^\sharp h_K, \qquad 
\| \Jac_K^{-1}  \|_{\ell^2} \leq c^\flat h_K^{-1},
\label{Eq2:propJK}
\end{equation}
where $h_K$ is the diameter of $K$. It can be shown that 
$c^\sharp=\frac{1}{\rho_{\wK}}$ and
$c^\flat = \frac{h_K}{\rho_K}h_{\wK}$ for meshes composed of
simplices, where $\rho_K$ is the diameter of the largest ball that can
be inscribed in $K$, $h_\wK$ is the diameter of $\wK$, and $\rho_\wK$
is the diameter of the largest ball that can be inscribed in $\wK$.

\subsection{Definition of $\delta(\bx)$}
In the arguments to follow, we are going to invoke smoothing operators
like those defined in \S\ref{Sec:Mollification}. To avoid having to
assume that the mesh sequence is quasi-uniform, we construct a
meshsize function $\frh\in C^{0,1}(\overline{\Dom};\Real)$ such that there are three uniform constants
$c,c',c''>0$ so that 
\begin{align} 
\|\frh\|_{W^{1,\infty}(\Dom;\Real)}\le c,\qquad  c' h_K \le \frh(\bx) \le c'' h_K,\quad  \forall \bx \in K,
\end{align}
for all $K\in\calT_h$. The construction of this function is standard
in the finite element literature.  For instance, if the mesh is
composed of simplices, consider the piecewise linear function whose
value at any vertex of the mesh is the average of the mesh-sizes of
the simplices sharing this vertex.

Following \citet{Christiansen_Winther_mathcomp_2008}, we introduce $\epsilon \in (0,1)$ and define
\begin{equation} \label{eq:delta_eps_frh}
\delta(\bx) := \epsilon \frh(\bx), \qquad \forall \bx \in \Dom.
\end{equation}
Then we can define $\bvarphi_\delta$ and $\bvartheta_\delta$ like in
\eqref{map_varphi_delta_Lipschitz} and
\eqref{map_vartheta_delta_Lipschitz}, and we can also define generic
mollifying operators $\calK_\delta$ and $\calK_{\delta,0}$ like in
\eqref{def_calK_delta} and \eqref{def_calK_delta_zero}.
Lemmas~\ref{Lem:shrinking_Lipschitz}\&\ref{Lem:enlarging_Lipschitz}
hold for $\ell\in\{0,1\}$ only, and the smoothness statement in
Lemmas~\ref{Lem:Kdelta_Cinfty}\&\ref{Lem:Kdelta_Cinfty_zero} must be
replaced by $\calK_{\delta} (g) \in C^1(\Dom;\Real^q)$ and
$\calK_{\delta,0} (g) \in C^1_0(\Dom;\Real^q)$ for all
$g\in L^1(\Dom;\Real^q)$, respectively, since $\delta$ is only
Lipschitz.  All the other statements in \S\ref{Sec:Mollification} and
\S\ref{Sec:Mollification_zero} remain unchanged.

\subsection{Reference and local finite elements}
We are going to consider various approximation spaces based on the
mesh sequence $\famTh$ and a fixed reference finite element $\wKPS$.
We henceforth assume that $\wP$ is composed of $\Real^q$-valued functions for some
integer $q\ge1$ and that $\wP\subset W^{1,\infty}(\wK;\Real^q)$
(recall that $\wP$ is a space of polynomial functions
in general).  The reference degrees of freedom and the associated
reference shape functions are denoted
$\{\wsigma_{1},\ldots, \wsigma_{\nf}\}$ and
$\{\wtheta_{1},\ldots, \wtheta_{\nf}\}$, respectively.  We denote
$\calN:=\intset{1}{\nf}$ to alleviate the notation.
We assume that the linear forms
$\{\wsigma_{i}\}_{i\in\calN}$ can be extended to
$\calL(V(\wK);\Real)$, where $V(\wK)$ is a Banach space such that
$V(\wK) \subset L^1(\wK;\Real^q)$; see~\citep[p.~39]{ErnGuermond_FEM}.
The interpolation operator $\inter_{\wK}:V(\wK)\to \wP$ associated
with the reference finite element $\wKPS$ is defined by
\begin{equation}
\inter_{\wK}(\wv)(\wbx) = 
\sum_{i\in \calN} \wsigma_i(v) \wtheta_i(\wbx), \qquad \forall \wbx\in
\wK,
\quad \forall \wv\in V(\wK).
\label{Eq2:OpIntLoc}
\end{equation}
By construction, $\inter_\wK\in\calL(V(\wK);\wP)$, and $\wP$ is
point-wise invariant by $\inter_{\wK}$.  

Let $K$ be a cell in the mesh $\calT_h$. We introduce a $q\CROSS q$
invertible matrix $\polA_K$ and define the mapping
$\mapK\in \calL(L^1(K;\Real^q);L^1(\wK;\Real^q))$ by
\begin{equation} \label{localization_of_mapK}
\mapK(v) = \polA_K (v\circ\trans_K).
\end{equation}
 It can be shown (see \cite[Prop.~1.61]{ErnGuermond_FEM}) that upon
setting
\begin{subequations}\label{Eq2:genEF}
\begin{align}
&P_K:=\bset  p=\mapKmun(\wp) \tq \wp \in \wP \eset,\label{Eq2:genEF_P}\\
&\Sigma_K:= \{\sigma_{K,i}\}_{i\in\calN} \; \text{s.t.} \;
\sigma_{K,i} = \wsigma_i\circ \mapK,\label{Eq2:genEF_S}
\end{align}
\end{subequations}
the triple $(K,P_K,\Sigma_K)$ is a finite element. Moreover, the interpolation
operator
\begin{equation}
  \inter_{K}(v)(\bx) = 
  \sum_{i\in\calN} \sigma_{K,i}(v) \theta_{K,i}(\bx), \qquad \forall \bx\in K,\quad \forall v\in V(K),
\label{def_of_interK}
\end{equation}
where we have set $\theta_{K,i} := \mapKmun(\wtheta_i)$, is such that
$\inter_K\in\calL(V(K);P_K)$ and $P_K$ is point-wise invariant by $\inter_K$.
Definition \eqref{Eq2:genEF_P} implies that
$P_K\subset W^{1,\infty}(K;\Real^q)$. More generally $\mapK$ maps
$W^{l,p}(K;\Real^q)$ to $W^{l,p}(\wK;\Real^q)$ for all $l\in\Natural$,
all $p\in [1,\infty]$ (with $z^{\pm \frac1p}=1$, $\forall z>0$ if
$p=\infty$) and
\begin{subequations}\label{eq:bnd_calL_psi}\begin{align}
|\mapK|_{\calL(W^{l,p}(K;\Real^q);W^{l,p}(\wK;\Real^q))} 
&\le c\, \|\polA_K\|_{\ell^2}\; \| \Jac_K \|_{\ell^2}^l  \; |{\det(\Jac_K)}|^{-\frac1p},\\
|\mapKmun|_{\calL(W^{l,p}(\wK;\Real^q);W^{l,p}(K;\Real^q))} 
&\le c\, \|\polA_K^{-1}\|_{\ell^2}\;\| \Jac_K^{-1} \|_{\ell^2}^l  \; |{\det(\Jac_K)}|^{\frac1p},\label{eq:bnd_calL_psi_mun}
\end{align}\end{subequations}
for all $K\in\calT_h$, (see \eg
\citep[Thm.~3.1.2]{Ciarlet_FE_Book_2002} or \citep[Lemma
1.101]{ErnGuermond_FEM}). 

\subsection{Structural assumptions}
We henceforth assume that there is a uniform
constant $c$ such that
\begin{equation}
  \|\polA_K\|_{\ell^2}\|\polA_K^{-1}\|_{\ell^2}\le c\, \|\Jac_K\|_{\ell^2}\|\Jac_K^{-1}\|_{\ell^2},
\label{Assumption:AkAkminusone_bounded}
\end{equation}
so that, owing to~\eqref{Eq2:propJK}, $\|\polA_K\|_{\ell^2}\|\polA_K^{-1}\|_{\ell^2}$ is uniformly bounded with respect to $K$ and $h$. 
We also assume that the degrees of freedom over $\wK$ are either point values
or integrals over edges, faces or $\wK$ itself.  This is formalized by
assuming that
\begin{equation}
|\wsigma_{i}(\wv)| \le c \begin{cases}
\|\wv(\wba_i)\|_{\ell^2(\Real^q)}  & \text{if point evaluation at $\wba_i$},\\
\frac{1}{\mes{\wS_{\wK,i}}}
\int_{\wS_{\wK,i}} \|\wv\|_{\ell^2(\Real^q)}\dif s, & \text{otherwise,}
\end{cases}
\label{Integral_bound_sigma_Ki}
\end{equation}
where $\wS_{\wK,i}$ is either an edge, a face, or $\wK$
itself. All these mesh-related geometric entities are assumed to
  be closed sets.

In the case of a point evaluation at $\wba_i$, we observe that
since the cardinal number of $\wSigma$ is finite, there exists a
distance $\wl_0>0$ such that only one of the following situations
occurs: (1) $\wba_i$ is a vertex of $\wK$; (2) $\wba_i$ is in the interior of an edge
of $\wK$ and is at least at distance $\wl_0$ from any vertex; (3)
$\wba_i$ is in the interior of a face of $\wK$ and is at least at distance $\wl_0$ from
any edge; (4) $\wba_i$ is in the interior of $\wK$ and is at least at
distance $\wl_0$ from any face (with the obvious extension in
higher space dimension).  

Let $K\in\calT_h$ and 
denote by $\{\ba_j\}_{j\in \calM_K}$ the collection of points
associated with the degrees of freedom in $K$ defined by point evaluation.
Note that there exists $\wba_i\in \wK$ such that
$\ba_j=\bT_K(\wba_i)$ for all $j\in \calM_K$.  
The shape-regularity of the mesh sequence implies that there is a
constant $c_{\min}$ (uniform with respect to $j$, $K$, and $\calT_h$)
such that the open ball $B(\ba_j,c_{\min} h_K)$ has the following
property: for all $K'$ such that
$K'\cap B(\ba_j,c_{\min} h_K) \ne \emptyset$ and every
$\bx\in K'\cap B(\ba_j, c_{\min} h_K)$, the entire segment
$[\bx,\ba_j]$ is in $K'$. An immediate consequence of this observation
is that
\begin{equation}
\|v(\bx) - v(\ba_j) \|_{\ell^2} \le \|\bx-\ba_j\|_{\ell^2} \|\GRAD v\|_{L^\infty(K';\Real^q)},
\quad \forall\bx\in K'\cap B(\ba_j,c_{\min} h_K) \ne \emptyset,
\label{control_on_vx_minus_vai}
\end{equation}
for all $v\in P_K$. Note that this implies that
$B(\ba_j,c_{\min} h_K)\subset \calT_K$.

In the rest of the paper, we define $\epsilon_{\max}>0$ such that
\begin{subequations}\begin{align}
&\max_{j\in \calM_K} \max_{\by\in B(\bzero,1)}\|\ba_j - (\bvarphi_{\delta(\ba_j)}(\ba_j) +  r \delta(\ba_j)  \by) \|_{\ell^2} 
\le c_{\min} h_K, \label{assumption1_on_epsilon_max} \\
&\cup_{\bx\in K} (\bvarphi_{\delta(\bx)}(\bx) +  r \delta(\bx)  B(\bzero,1) ) \subset \Dom_K,
\label{assumption2_on_epsilon_max} 
\end{align}\end{subequations}
for all $K\in \calT_h$, all $h>0$, and all functions $\delta$ satisfying~\eqref{eq:delta_eps_frh} for any $\epsilon \in (0,\epsilon_{\max}]$.


\subsection{Finite element spaces}
We introduce the broken finite element space
\begin{equation} \label{eq:PTh_b} 
P\upb(\calT_h) = \bset v_h\in
   L^1(\Dom;\Real^q) \tq \mapK(v_{h|K})\in \wP,\, \forall K\in\calT_h\eset,
\end{equation}
where the statement $\mapK(v_{h|K})\in \wP$
is equivalent to $v_{h|K} \in P_K$. 
Notice also that $P\upb(\calT_h) \subset W^{1,\infty}(\calT_h;\Real^q):=
\bset v\in L^\infty(\Dom;\Real^q)\st v_{|K}\in W^{1,\infty}(K;\Real^q),\ \forall
K\in \calT_h\eset$ since $P_K\subset W^{1,\infty}(K;\Real^q)$. We denote by $\inter_h\upb:L^p(\Dom)\to P\upb(\calT_h)$ the interpolation operator such that $\inter_h\upb(v)_{|K}=\inter_K(v_{|K})$,
for all $K\in \calT_h$.

We now introduce the notion of interfaces and jump across interfaces.
We say that a subset $F\subset\overline\Dom$ with a positive
$(d{-}1)$-dimensional measure is an interface if there are distinct
mesh cells $K_l,K_r\in\calT_h$ such that
$F=\partial K_l\cap \partial K_r$. We say that a subset
$F\subset\overline\Dom$ with positive $(d{-}1)$-dimensional measure is
a boundary face if there is a mesh cell $K\in\calT_h$ such that
$F=\partial K\cap \partial \Dom$.  The unit normal vector $\bn_F$ on
$F$ is conventionally chosen to point from $K_l$ to $K_r$ for an
interface and to point outward for a boundary face.  The interfaces are
collected in the set $\calFhi$, the boundary faces  are collected in the set $\calFhb$,
and we let $\calFh= \calFhi\cup \calFhb$.  Let $F\in\calFhi$ be a mesh
interface, and let $K_l,K_r$ be the two cells such that
$F=\partial K_l\cap \partial K_r$; the jump of
$v\in W^{1,1}(\calT_h;\Real^q)$ across $F$ is defined to be
\begin{equation}
\label{Eq:def_jump_scal}
\jump{v}_F(\bx) =v_{|K_l}(\bx) - v_{|K_r}(\bx) \qquad \text{\ae}\ \bx \in F.
\end{equation}

Next we asume to have at hand a Banach space $W\subset L^1(\Dom;\Real^q)$, with
continuous embedding,
where some notion of jump
across interfaces makes sense.  More precisely, we assume that there is a
(bounded) linear trace operator
$\gamma_K: W^{1,1}(K;\Real^q) \longrightarrow L^1(\partial
K;\Real^t)$,
for some $t\ge1$ and for all $K\in \calT_h$, and we define the notion
of $\gamma$-jump across interfaces as follows:
\begin{equation}
\label{Eq:def_gamma_jump}
\jump{v}_F^\gamma(\bx) =\gamma_{K_l}(v_{|K_l})(\bx) - \gamma_{K_r}(v_{|K_r})(\bx) 
\qquad \text{\ae}\ \bx \in F.
\end{equation}
We assume that
$|\jump{v}_F^\gamma(\bx)|\le|\jump{v}_F(\bx)|$, for \ae $\bx\in F$, for all
$v\in W^{1,1}(\calT_h)$, so that
\begin{equation}
\label{eq:gamma_trace_0}
v\in W^{1,1}(\Dom;\Real^q) \implies  (\jump{v}_F^\gamma =
0,\;\forall F\in \calFhi).
\end{equation}
We relate the notion of $\gamma$-jump to the space $W$ by 
assuming that
\begin{equation}
v\in W\cap W^{1,1}(\calT_h;\Real^q) \implies  (\jump{v}_F^\gamma =
0,\;\forall F\in \calFhi),
\end{equation}
and, conversely, that a function in $W^{1,\infty}(\calT_h;\Real^q)$ with zero
$\gamma$-jumps across interfaces is in $W$.
With this setting, we define 
\begin{equation} \label{eq:def_p_v}
P(\calT_h) := P\upb(\calT_h) \cap W.
\end{equation}
The above assumptions imply that 
\begin{equation}
P(\calT_h) = \bset v_h\in
 P\upb(\calT_h)  \tq \jump{v_h}_F^\gamma=0,\ \forall F\in\calFhi\eset.
\label{characterization_of_PcalTh_by_jumps}
\end{equation}

Let $F\in \calFhb$ be a boundary face and denote by $K_F$ the unique cell
such that $F\subset \partial K_F$. We consider the global trace operator $\gamma:
W^{1,1}(\Dom;\Real^q) \longrightarrow L^1(\front;\Real^t)$ such that  
\begin{equation}
\gamma(v)_{|F} = \gamma_{K_F}(v_{|{K_F}}), \qquad \forall F\in \calFhb.
\end{equation}
We assume that $\gamma$ can be extended to $W$ into a bounded linear
operator $\gamma: W\longrightarrow W^\partial$ where $W^\partial$ is
an appropriate Banach space, whose exact structure is not important
for the time being. We define $W_0 = \ker(\gamma)$, \ie
$W_0 = \{v\in W \st \gamma(v) =0\}$.  Let us introduce
$P_0(\calT_h) =P(\calT_h)\cap W_0$:
\begin{equation}
P_0(\calT_h) := \{v_h\in P(\calT_h)\tq \gamma(v_h)=0\}. \label{def_of_Ph0}
\end{equation}


\subsection{Examples}\label{Sec:Examples} 
The present theory covers a large class of scalar- and vector-valued
finite elements like Lagrange, N\'ed\'elec, and Raviart-Thomas
finite elements. To remain general, we denote the three reference elements
corresponding to the above three classes as follows:
$(\wK,\wP\upg,\wSigma\upg)$, $(\wK,\wbP\upc,\wSigma\upc)$ and
$(\wK,\wbP\upd,\wSigma\upd)$. The corresponding domains for the
degrees of freedom are denoted $V\upg(\wK)$, $\bV\upc(\wK)$,
$\bV\upd(\wK)$. We think of $(\wK,\wP\upg,\wSigma\upg)$ as a
scalar-valued finite element ($q=1$) and some of its degrees of freedom
require point evaluation, for instance $(\wK,\wP\upg,\wSigma\upg)$ could be a
Lagrange element. We assume that the finite element
$(\wK,\wbP\upc,\wSigma\upc)$ is vector-valued ($q=d$) and some of its
degrees of freedom require to evaluate integrals over edges. Typically,
$(\wK,\wbP\upc,\wSigma\upc)$ is a N\'ed\'elec-type or edge element.
Likewise, the finite element $(\wK,\wbP\upd,\wSigma\upd)$ is assumed
to be vector-valued ($q=d$) and some of its degrees of freedom are
assumed to require evaluation of integrals over
faces. Typically, $(\wK,\wbP\upd,\wSigma\upd)$ is a
Raviart-Thomas-type element.  The arguments developed herein 
do not require to know the exact structure of the above
elements.

The above assumptions imply that it is admissible to choose
$V\upg(\wK) = W^{s,p}(\wK)$ with $s>\frac{d}{p}$,
$\bV\upc(\wK) = \bW^{s,p}(\wK)$ with $s>\frac{d-1}{p}$, and
$\bV\upd(\wK) = \bW^{s,p}(\wK)$ with $s>\frac{1}{p}$ (recall that
denoting by $M$ a smooth
manifold of dimension $d'$ in $\wK$, the restriction
operator to $M$ is continuous from $W^{s,p}(\wK)$ to $L^p(M)$
provided $s>\frac{d-d'}{p}$). 
Note that it
is also legitimate to choose 
\[
V\upg(\wK) = W^{d,1}(\wK),\quad
\bV\upd(\wK) = \bW^{1,1}(\wK), \quad
\bV\upc(\wK) = \bW^{d-1,1}(\wK),
\]
since $W^{d,1}(\wK)\subset C^0(\wK)$), functions in $W^{1,1}(\wK)$
have a trace in $L^1(\partial \wK)$, and functions in $W^{d-1,1}(\wK)$
have integrable traces on the one-dimensional edges of $\wK$.


Let $\mapKg$, $\mapKc$, $\mapKd$ be the linear maps introduced in
\eqref{localization_of_mapK} for each of the reference finite elements
defined above. In practice $\mapKg$ is the pullback by $\trans_K$, and
$\mapKc$ and $\mapKd$ are the contravariant and covariant
Piola transformations, respectively, \ie
\begin{subequations}\begin{align}
\polA_K\upg &= 1,
&& \mapKg(v) =  v\circ\trans_K, \label{Eq:Def_mapg} \\
\polA_K\upc &=\Jac_K\tr,
&&\mapKc(\bv) = \Jac_K\tr (\bv\circ\trans_K), \label{Eq:def_Piola_rot}\\
\polA_K\upd &=\det(\Jac_K)\,\Jac_K^{-1},
&& \mapKd(\bv) =  \det(\Jac_K)\,\Jac_K^{-1}(\bv\circ\trans_K).\label{Eq:Def_Piola} 
\end{align}\end{subequations}
Note that $c=1$ in \eqref{Assumption:AkAkminusone_bounded} for the above examples. 
We consider the following $\gamma$-traces:
\begin{subequations}\begin{alignat}{2}
\gamma_{K}\upg(v_{|K})(\bx) &:= v_{|K}(\bx),&\quad& \forall \bx\in F,\\
\bgamma_{K}\upc(\bv_{|K})(\bx) &:= \bv_{|K}(\bx)\CROSS \bn_F,&\quad& \forall \bx\in F,\\
\bgamma_{K}\upd(\bv_{|K})(\bx)&:= \bv_{|K}(\bx)\SCAL\bn_F,&\quad &\forall \bx\in F,
\end{alignat}\end{subequations}
and the following finite element spaces:
\begin{subequations}\begin{align}
P\upg(\calT_h)    &:= 
\{v_h \in L^1(\Dom) \st \psi_K\upg(v_{h|K}) {\,\in\,} \wP\upg, \ \forall K\in \calT_h, \
\jump{v_h}_F\upg=0,\ \forall F\in\calFhi\}, \\
\bP\upc(\calT_h)    &:= 
\{\bv_h \in \bL^1(\Dom) \st \bpsi_K\upc(\bv_{h|K}) {\,\in\,} \wbP\upc, \ \forall K\in \calT_h, \
\jump{\bv_h}_F\upc=\bzero,\ \forall F\in\calFhi\}, \hspace{-2pt}\\
\bP\upd(\calT_h)    &:= 
\{\bv_h \in \bL^1(\Dom) \st \bpsi_K\upd(\bv_{h|K}) {\,\in\,} \wbP\upd, \ \forall K\in \calT_h, \
\jump{\bv_h}_F\upd=0,\ \forall F\in\calFhi\},\hspace{-2pt}
\end{align} \end{subequations}
where we simplified the notation by using $\jump{v_h}_F\upg$
instead of $\jump{v_h}_F^{\gamma\upg}$, \etc 
Note the conformity properties $P\upg(\calT_h)\subset Z\upgp(\Dom)$,
$\bP\upc(\calT_h)\subset \bZ\upcp(\Dom)$, and
$\bP\upd(\calT_h)\subset \bZ\updp(\Dom)$.
Likewise, observing that
$Z\upgp_0(\Dom) := \bset v\in Z\upgp(\Dom)\st  \gamma\upg(v)=0\eset$ \etc, we define
\begin{subequations}\begin{align}
P_0\upg(\calT_h) &:= P\upg(\calT_h)\cap Z_0\upgp(\Dom),\\
\bP_0\upc(\calT_h) &:= \bP\upc(\calT_h)\cap \bZ_0\upcp(\Dom),\\
\bP_0\upd(\calT_h) &:= \bP\upd(\calT_h)\cap \bZ_0\updp(\Dom).
\end{align} \end{subequations}

We finally denote by $\inter_h\upg$, $\inter_h\upd$, $\inter_h\upc$
and $\inter_{h0}\upg$, $\inter_{h0}\upd$, $\inter_{h0}\upc$ the
canonical interpolation operators associated with the finite element
spaces $P\upg(\calT_h) $, $\bP\upc(\calT_h)$, $\bP\upd(\calT_h)$, and
$P_0\upg(\calT_h) $, $\bP_0\upc(\calT_h)$, $\bP_0\upd(\calT_h)$.
Note that $\inter\upg_h : W^{s,p}(\Dom) \longrightarrow
P\upg(\calT_h)\subset L^p(\Dom)$ is stable provided $s> \frac{d}{p}$, $\IN_h:
\bW^{s,p}(\Dom) \longrightarrow \bP\upc(\calT_h)\subset \bL^p(\Dom)$
is stable provided $s> \frac{d-1}{p}$ ($=\frac{2}{p}$ for $d=3$), and $\IRT_h:
\bW^{s,p}(\Dom) \longrightarrow \bP\upd(\calT_h)\subset \bL^p(\Dom)$
is stable provided $s>\frac{d-(d-1)}{p}=\frac{1}{p}$.
We finally assume that the polynomial degrees in each of these spaces
are compatible so that the following commuting properties hold with
$s>\frac{d}{p}$:
\begin{equation}
\begin{diagram}[height=1.7\baselineskip,width=1.5cm]
V\upg(\Dom)  & \rTo^{\GRAD} & \bV\upc(\Dom)  & \rTo^{\ROT} & \bV\upd(\Dom)  & \rTo^{\DIV} & L^{p}(\Dom)  \\
\dTo_{\inter_h\upg} & & \dTo_{\IN_h} & & \dTo_{\IRT_h}& & \dTo_{\inter_h\upb} \\
P\upg(\calT_h) & \rTo^{\GRAD} & \bP\upc(\calT_h) & \rTo^{\ROT} & \bP\upd(\calT_h) & \rTo^{\DIV} & P\upb(\calT_h) 
\end{diagram}
\label{Diag:V_Ph}
\end{equation}
where 
\begin{subequations}\begin{align}
V\upg(\Dom) &=\{f\in W^{s,p}(\Dom)\st \GRAD f\in \bW^{s-\frac1p,p}(\Dom)\}, \label{def:Vg}\\
\bV\upc(\Dom) &=\{\bg\in \bW^{s-\frac1p,p}(\Dom)\st \ROT \bg\in \bW^{s-\frac2p,p}(\Dom)\}, \label{def:Vc}\\
\bV\upd(\Dom) &=\{\bg\in \bW^{s-\frac2p,p}(\Dom)\st \DIV \bg\in W^{s-\frac3p,p}(\Dom)\},\label{def:Vd}
\end{align}\end{subequations}
and $\inter_h\upb$ is an interpolation operator only involving integrals over mesh cells.
Likewise, upon introducing
\begin{subequations}\begin{align}
V\upg_0(\Dom)&=\{f\in V\upg(\Dom) \st f_{|\front}=0\},\\
\bV\upc_0(\Dom)&=\{\bg\in \bV\upc(\Dom) \st \bg\CROSS \bn_{|\front}=\bzero\},\\
\bV\upd_0(\Dom)&=\{\bg\in \bV\upd(\Dom) \st \bg\SCAL\bn_{|\front}=0\},
\end{align}\end{subequations}
we assume that the following diagram commutes:
\begin{equation}
\begin{diagram}[height=1.7\baselineskip,width=1.5cm]
V\upg_0 (\Dom) & \rTo^{\GRAD} & \bV\upc_0(\Dom) & \rTo^{\ROT} &\bV\upd_0(\Dom) & \rTo^{\DIV} & L^p(\Dom) \\
\dTo_{\inter_{h0}\upg} & & \dTo_{\IN_{h0}} & & \dTo_{\IRT_{h0}}& & \dTo_{\inter_{h}\upb} \\
P\upg_0(\calT_h) & \rTo^{\GRAD} & \bP\upc_0(\calT_h) & \rTo^{\ROT} & \bP\upd_0(\calT_h) & \rTo^{\DIV} & P\upb(\calT_h) 
\end{diagram}\label{Diag:V0_Ph0}
\end{equation}

\section{Stable, commuting, 
quasi-interpolation projection}
\label{sec:quasi.inter}
We introduce in this section a family of finite-element-based
quasi-interpolation operators (with and without boundary conditions)
that are $L^p$-stable, commute with the standard
differential operators $\GRAD$, $\ROT$, and $\DIV$, and preserve
the above finite element spaces.

\subsection{The operator $\inter_h\calK_\delta$}
Owing to the properties of the smoothing operators established above,
it makes sense to consider the discrete functions
$\inter\upg_h\calK_\delta\upg f \in P\upg(\calT_h)$,
$\inter\upg_{h0}\calK_\delta\upg f \in P_0\upg(\calT_h)$,
$\inter\upb_h\calK_\delta\upb f \in P\upb(\calT_h)$,
$\inter\upc_h\calK_\delta\upc \bg \in \bP\upc(\calT_h)$,
$\inter\upc_{h0}\calK_\delta\upc \bg \in \bP_0\upc(\calT_h)$,
$\inter\upd_h\calK_\delta^d \bg \in \bP\upd(\calT_h)$ and
$\inter\upd_{h0}\calK_\delta^d \bg \in \bP_0\upd(\calT_h)$ for any
integrable scalar-valued function $f$ and any integrable vector-valued
function $\bg$.  We now establish some stability properties of the
restrictions of the operators $\inter\upg_h\calK_\delta\upg$,
$\inter\upg_{h0}\calK_\delta\upg$, $\inter\upb_h\calK_\delta\upb$,
$\inter\upc_h\calK_\delta\upc$, $\inter\upc_{h0}\calK_\delta\upc$,
$\inter\upd_h\calK_\delta^d$, and $\inter\upd_{h0}\calK_\delta^d$ to
the discrete spaces $P\upg(\calT_h)$, $P_0\upg(\calT_h)$, $P\upb(\calT_h)$,
$\bP\upc(\calT_h)$, $\bP_0\upc(\calT_h)$, $\bP\upd(\calT_h)$,
and $\bP_0\upd(\calT_h)$.

To avoid repeating proofs seven times, we denote by $\calI_h$ one of the
interpolation operators introduced above and $\calK_\delta$ the
corresponding smoothing operator; likewise, the range of $\calI_h$ is
denoted $P(\calT_h)$.  We assume that $P(\calT_h)$ is composed of
$\Real^q$-valued fields. 

\begin{remark}[Boundary conditions]
  Note that we do not invoke $\calK_{\delta,0}\upg$,
  $\calK_{\delta,0}\upc$, and $\calK_{\delta,0}\upd$ in the above
  construction.  The theory to be exposed in the next section holds
  by using $\calK_{\delta}\upg$, $\calK_{\delta}\upc$, and
  $\calK_{\delta}\upd$ in all the cases, whether homogeneous boundary
  conditions are enforced or not in the discrete spaces.
\end{remark} 

\subsection{$L^p$-stability of the operator $\inter_h\calK_\delta$}
We start
with a key result in the spirit of
\cite[Lem.~4.2]{Christiansen_Winther_mathcomp_2008}, see also \cite[Lem.~6]{Schoberl:05}.  This result is crucial to devise a quasi-interpolation operator that preserves the finite element space $P(\calT_h)$.
\begin{lemma}[Discrete $L^p$-approximation] 
  \label{Lem_IL_Kdelta_Lp_stability} 
  There is $c_{\text{\rm stab}}>0$, uniform with respect to the mesh
  sequence, such that
  $\|f_h - \inter_h\calK_\delta f_h\|_{L^p(\Dom;\Real^q)}\le c_{\text{\rm stab}}
  \epsilon \|f_h\|_{L^p(\Dom;\Real^q)}$
  for all $\epsilon \in (0,\epsilon_{\max}]$, all $f_h\in P(\calT_h)$
  and all $p\in [1,\infty]$.
\end{lemma}
\begin{proof} 
 (1)  Let $f_h\in P(\calT_h)$ and let us set
  $e_h:= f_h - \inter_h \calK_\delta f_h$ and
  $e:= f_h - \calK_\delta f_h$; note that
  $e_h = \inter_h e$. Let $K$ be a cell in $\calT_h$, then
  using that $\theta_{K,i} := \psi_K^{-1}(\wtheta_i)$, we have
  $\|\theta_{K,i}\|_{L^p(K;\Real^q)} \le \det(\Jac_K)^{\frac1p}
  \|\polA_K^{-1}\|_{\ell^2} \|\wtheta_{i}\|_{L^p(\wK;\Real^q)}$
  for all $i\in \intset{1}{\nf}$, and we infer that
\begin{align}
\|e_h\|_{L^p(K;\Real^q)} = \|\inter_h e\|_{L^p(K;\Real^q)} 
&\le \sum_{i\in\calN} |\sigma_{K,i}(e)| \|\theta_{K,i}\|_{L^p(K;\Real^q)}\nonumber\\
&\le  \det(\Jac_K)^{\frac1p}\|\polA_K^{-1}\|_{\ell^2} \sum_{i\in\calN}|\sigma_{K,i}(e)|, \label{eq:bnd_eh}
\end{align}
The rest of the proof consists of estimating $\sigma_{K,i}(e)$. 

(2) Let us assume first that the degree of freedom $\sigma_{K,i}$ is a
value at a point $\ba_j:=\bT_K(\wba_i)$ in $K$.  Then using the
assumption~\eqref{Integral_bound_sigma_Ki} and the definition~\eqref{localization_of_mapK} of $\mapK$, we infer that
$|\sigma_{K,i}(e)| \le c\, \|\polA_K\|_{\ell^2} \|e(\ba_j)\|_{\ell^2}$.
By proceeding as in the proof of
Theorem~\ref{Thm:Kdeltaf_to_f_in_LP} (step (2)), we obtain
\begin{align*}
e(\ba_j)= f_h(\ba_j) - \calK_\delta f_h(\ba_j) = \int_{B(\bzero,1)} \rho(\by)
\big(f_h(\ba_j) - f_h(\bvarphi_{\delta(\ba_j)}(\ba_j)+ \delta(\ba_j) \radius \by) \big)\dif y.
\end{align*}
Owing to \eqref{control_on_vx_minus_vai} and
\eqref{assumption1_on_epsilon_max} (recall that $\epsilon\le
\epsilon_{\max}$), we have
\begin{align*}
\|e(\ba_j)\|_{\ell^2} &\le c \max_{\by\in B(\bzero,1)} 
\|f_h(\ba_j) - f_h(\bvarphi_{\delta(\ba_j)}(\ba_j)+ \delta(\ba_j) \radius \by)\|_{\ell^2}\\
&\le c' \delta(\ba_j)\max_{K'\in \calT_K}  \|\GRAD f_h\|_{\bL^\infty(K';\Real^q)}  
\le c'' \epsilon h_K \max_{K'\in \calT_K}\|\GRAD f_h\|_{\bL^\infty(K';\Real^q)}.
\end{align*}
Finally using a local inverse inequality, which is legitimate since
the mesh sequence is shape-regular, we infer that $|\sigma_{K,i}(e)|
\le c\, \epsilon \|\polA_K\|_{\ell^2}
\|f_h\|_{L^\infty(\Dom_K;\Real^q)}$.  Note that the purpose of the
above argument is to account for the fact that $f_h$ is (a priori)
only piecewise Lipschitz (\ie can be discontinuous across interfaces)
but $f_h$ is necessarily continuous at $\ba_j$.

(3) If the degree of freedom $\sigma_{K,i}$ is an integral over an
edge, face or over $K$, we use \eqref{Integral_bound_sigma_Ki}, \ie
$|\sigma_{K,i}(e)| \le c\|\polA_K\|_{\ell^2}\frac{1}{\mes{S_{K,i}
  }} \int_{S_{K,i}} \|e\|_{\ell^2} \dif s$.
We define
$\calT_{S_{K,i}} =\{K'\in \calT_K \st {S_{K,i}} \subset K'\}$ and we
introduce
$S_{K,i}\upint = \{ \bx \in S_{K,i} \st \bvarphi_{\delta(\bx)}(\bx)
+\delta(\bx)\radius B(\bzero,1) \subset \calT_{S_{K,i}}\}$
and $S_{K,i}\upbnd = S_{K,i}{\setminus} S_{K,i}\upint$. Then using
\eqref{assumption2_on_epsilon_max} and setting
$\bpsi_\delta(\bx,\by)=\bvarphi_{\delta(\bx)}(\bx) +\delta(\bx)\radius\by$, we have
\begin{align*}
  \int_{S_{K,i}\upint} \|e\|_{\ell^2} \dif s &\le 
  \int_{S_{K,i}\upint} \int_{B(\bzero,1)} \rho(\by)
  \left\|f_h(\bpsi_\delta(\bx,\by)) - f_h(\bx)\right\|_{\ell^2} \dif y\dif s \\
  &\le  \int_{S_{K,i}\upint} \sum_{K'\in \calT_{S_{K,i}}
  }\int_{\aatop{\by\in B(\bzero,1)}{\bpsi_\delta(\bx,\by)\in K'}}
  \left\|f_h(\bpsi_\delta(\bx,\by)) -
    f_h(\bx)\right\|_{\ell^2} \dif s\dif y \\
  & \le c \mes{S_{K,i}} \epsilon h_K \sum_{K'\in \calT_{S_{K,i}}}\|\GRAD f_h\|_{\bL^\infty(K';\Real^{q})} 
\le c \mes{S_{K,i}} \epsilon \|f_h\|_{L^\infty(\Dom_K;\Real^q)},
\end{align*}
where we used the shape-regularity of the mesh sequence (\ie $h_{K'}
\le c h_K$) and an inverse inequality.  Note again that the above
construction is meant to account for the fact that $f_h$ is (a priori)
only piecewise Lipschitz.  Moreover, if $\bx\in S_{K,i}\upbnd$, then
there is $\by\in B(\bzero,1)$ such that
$\bz:=\bvarphi_{\delta(\bx)}(\bx) + \delta(\bx)r\by$ is not in
$\calT_{S_{K,i}}$; then mesh-regularity implies that $ c\,
\dist(\bx,\partial S_{K,i}) \le \|\bz-\bx\|_{\ell^2}$ and that
$\|\bz-\bx\|_{\ell^2} \le c \delta(\bx) \le c' \epsilon
h_K$. Combining these bounds, we obtain that $|S_{K,i}\upbnd| \le
c\epsilon h_K |\partial S_{K,i}| \le c' \epsilon |S_{K,i}|$ (with the
convention that the $0$-dimensional measure of a point is $1$). As a
result, we infer that
\begin{align*}
  \int_{S_{K,i}\upbnd} \|e\|_{\ell^2} \dif s & \le
  \int_{S_{K,i}\upbnd} (\|f_h\|_{\ell^2}+ \|\calK_\delta f_h\|_{\ell^2})\dif s  \\
& \le c
  \|f_h\|_{L^\infty(\Dom_K;\Real^q)} \mes{S_{K,i}\upbnd} 
\le c'  \epsilon\, \mes{S_{K,i}} \|f_h\|_{L^\infty(\Dom_K;\Real^q)}.
\end{align*}
Combining the above two estimates yields
$|\sigma_{K,i}(e)| \le c\, \epsilon
\|\polA_K\|_{\ell^2}\|f_h\|_{L^\infty(\Dom_K;\Real^q)}$.

(4) We have established that $|\sigma_{K,i}(e)| \le c\, \epsilon
\|\polA_K\|_{\ell^2}\|f_h\|_{L^\infty(\Dom_K;\Real^q)}$ for all
possible degrees of freedom. Using the fact that
$\|\polA_K\|_{\ell^2}\|\polA_K^{-1}\|_{\ell^2}$ is uniformly bounded
together with an inverse inequality from $L^\infty(\Dom_K;\Real^q)$ to
$L^p(\Dom_K;\Real^q)$, we deduce that
\begin{align*}
  \|f_h -\calI_h\calK_\delta f_h\|_{L^p(K;\Real^q)} & =\|e_h\|_{L^p(K;\Real^q)} \\
&\le c\, \epsilon\det(\Jac_K)^{\frac1p}\|\polA_K^{-1}\|_{\ell^2} \|\polA_K\|_{\ell^2} \|f_h\|_{L^\infty(\Dom_K;\Real^q)} \\
&\le c\, \epsilon\, \|f_h\|_{L^p(\Dom_K;\Real^q)}.
\end{align*} 
We infer the desired result by summing over $K\in\calT_h$ and by invoking
the shape-regularity of the mesh sequence.
\end{proof}

The above lemma implies that
$\|(\polI - \inter_h\calK_\delta)_{|P(\calT_h)}\|_{\calL(L^p;L^p)}\le
c_{\text{\rm stab}} \epsilon$
for all $\epsilon \in (0,\epsilon_{\max}]$.  From now on we choose
$\epsilon$ once and for all by setting $\epsilon=\epsilon_{\min}$
with $\epsilon_{\min} := \min(\epsilon_{\max},(2c_{\text{\rm stab}})^{-1})$. Lemma~\ref{Lem_IL_Kdelta_Lp_stability} then implies that
\begin{equation}
\|(\polI -
\inter_h\calK_\delta)_{|P(\calT_h)}\|_{\calL(L^p;L^p)}\le \frac12.
\end{equation} 
This proves that $\inter_h\calK_{\delta|P(\calT_h)}$ is invertible for this
particular choice of $\epsilon$. Let $J_h: P(\calT_h) \longrightarrow P(\calT_h)$ be
the inverse of $ \inter_h\calK_{\delta|P(\calT_h)}$, \ie
\begin{equation}
  J_h  \inter_h\calK_{\delta|P(\calT_h)} = \inter_h\calK_{\delta|P(\calT_h)} J_h = \polI. \label{def_of_Jh}
\end{equation}
Note that the definition of $J_h$ implies that $\|J_h\|_{\calL(L^p;L^p)} \le 2$.

\begin{lemma}[$L^p$-stability] \label{Lem:Lp_stability_Ih_Kdelta} Let $\epsilon=\epsilon_{\min}$. 
  There is $c(\epsilon_{\min})$, uniform with respect to $h$, such that the following estimate holds:
  $\|\calI_h\calK_\delta\|_{\calL(L^p;L^p)}\le c(\epsilon_{\min})$.
\end{lemma}
\begin{proof}
 Let $f\in L^{p}(\Dom;\Real^q)$ and
  assume $p<\infty$. Then
\begin{align*}
\|\calI_h\calK_\delta f \|_{L^p(\Dom;\Real^q)}^p 
& = \sum_{K\in\calT_h} \int_K \left\|\sum_{i\in\calN} \sigma_{K,i}(\calK_\delta f) \theta_{K,i}(\bx)\right\|_{\ell^2}^p\dif x\\
&\le c \sum_{K\in\calT_h} \int_K \sum_{i\in\calN} |\sigma_{K,i}(\calK_\delta f)|^p \|\theta_{K,i}(\bx)\|_{\ell^2}^p\dif x.
\end{align*}
Using \eqref{Integral_bound_sigma_Ki}, we infer that
\begin{align*}
  \|\calI_h\calK_\delta f \|_{L^p(\Dom;\Real^q)}^p 
  &\le c \sum_{K\in\calT_h} \sum_{i\in\calN} \|\polA_K\|_{\ell^2}^p\|\calK_\delta f\|_{L^\infty(K;\Real^q)}^p 
    \|\polA_K^{-1}\|_{\ell^2}^p \mes{K},
\end{align*}
since $|\sigma_{K,i}(\calK_\delta f)| \le c\,\|\polA_K\|_{\ell^2}\|\calK_\delta f\|_{L^\infty(K;\Real^q)}$ and $\|\theta_{K,i}(\bx)\|_{L^\infty(K;\Real^q)}\le c\,\|\polA_K^{-1}\|_{\ell^2}$.
We conclude by invoking
Lemma~\ref{Lem:Local_inverse_ineq_for_Kdelta} below. The argument for
$p=\infty$ is similar.
\end{proof}

\begin{lemma}[Local inverse inequality] \label{Lem:Local_inverse_ineq_for_Kdelta} Let
  $\epsilon=\epsilon_{\min}$.  There is a uniform constant $c>0$ such
  that
\begin{equation}
\|\calK_\delta f \|_{L^\infty(K;\Real^q)} \le c  \epsilon_{\min }^{-d} \mes{K}^{-\frac1p} \|f\|_{L^p(\Dom_K;\Real^q)},
\end{equation}
for all $K\in \calT_h$, all $h>0$, and all $f\in L^p(\Dom;\Real^q)$. 
\end{lemma}
\begin{proof}
Let $\bx\in K$.
Since the function $\rho$ is bounded, we infer that
\[
\|\calK_\delta f (\bx)\|_{\ell^2} \le  c \int_{B(\bzero,1)} \|f(\bvarphi_{\delta(\bx)}(\bx) +\delta(\bx)\radius \by)\|_{\ell^2} \dif y.
\]
The condition \eqref{assumption2_on_epsilon_max} implies that 
\[
\|\calK_\delta f (\bx)\|_{\ell^2} \le c\, \|\delta^{-1}\|_{L^\infty(\Dom_K)}^{d} \int_{\Dom_K} \|f(\bz)\|_{\ell^2} \dif z
\le c\, \epsilon_{\min}^{-d} h_K^{-d} \mes{\Dom_K}^{1-\frac1p} \|f\|_{L^p(\Dom_K;\Real^q)},
\]
and we conclude using the shape-regularity of the mesh sequence.
\end{proof}

\subsection{Main result}
We now define the following operator
\begin{align}
\calJ_h = J_h \inter_h \calK_\delta,
\end{align}
and we state the main result of this section.
\begin{theorem}[Properties of $\calJ_h$] \label{Th:Pih} The following properties hold:
\berom
\item $P(\calT_h)$ is point-wise invariant under $\calJ_h$;
\item \label{Th:Pih_item2} There is $c$, uniform with respect to $h$, such that
  $\|\calJ_h\|_{\calL(L^p;L^p)}\le c$ and
\[
\|f - \calJ_h f \|_{L^p(\Dom;\Real^q)} \le c \inf_{f_h \in P(\calT_h)} \|f -
f_h\|_{L^p(\Dom;\Real^q)},
\]
for all $f\in L^p(\Dom;\Real^q)$;
\item \label{Th:Pih_item3} $\calJ_h$ commutes with the standard
  differential operators, \ie the following diagrams are commutative:
\eerom
\begin{equation}
\begin{diagram}[height=1.7\baselineskip,width=1.5cm]
Z\upgp (\Dom) & \rTo^{\GRAD} & \bZ\upcp(\Dom) & \rTo^{\ROT} & \bZ\updp(\Dom) & \rTo^{\DIV} &L^p(\Dom)\\
\dTo_{\calJ\upg_h} & & \dTo_{\calJ\upc_h} & & \dTo_{\calJ\upd_h}& & \dTo_{\calJ\upb_h} \\
P\upg(\calT_h) & \rTo^{\GRAD} & \bP\upc(\calT_h) & \rTo^{\ROT} &
\bP\upd(\calT_h) & \rTo^{\DIV} & P\upb(\calT_h)  
\end{diagram}
\end{equation}
\begin{equation}
\begin{diagram}[height=1.7\baselineskip,width=1.5cm]
Z\upgp_0 (\Dom) & \rTo^{\GRAD} & \bZ\upcp_0(\Dom) & \rTo^{\ROT} & \bZ\updp_0(\Dom) & \rTo^{\DIV} &L^p(\Dom)\\
\dTo_{\calJ_{h0}\upg} & & \dTo_{\calJ_{h0}\upc} & & \dTo_{\calJ_{h0}\upd}& & \dTo_{\calJ_{h}\upb} \\
P\upg_0(\calT_h) & \rTo^{\GRAD} & \bP\upc_0(\calT_h) & \rTo^{\ROT} & \bP_0\upd(\calT_h) & \rTo^{\DIV} & P\upb(\calT_h) 
\end{diagram}
\end{equation}
\end{theorem}
\begin{proof}
  The first property is a consequence of $\calJ_{h|P(\calT_h)}=\polI$, since
  $\calJ_h{\circ}\calJ_h =\calJ_{h|P(\calT_h)}{\circ}\calJ_h= \calJ_h$. The second
  property is proved by observing that the $L^p$-operator-norm of
  $J_h$ is bounded by $2$ and that of $\calI_h\calK_\delta$ is also
  uniformly bounded, as established in
  Lemma~\ref{Lem:Lp_stability_Ih_Kdelta}, since $\epsilon$ is now a
  fixed real number. Moreover, using that $\calJ_h
  f_h=f_h$ for all $f_h\in P(\calT_h)$, we have
\begin{multline*}
\|f - \calJ_h f \|_{L^p(\Dom;\Real^q)} =  \inf_{f_h \in P(\calT_h)} \|f - f_h
- \calJ_h (f - f_h)\|_{L^p(\Dom;\Real^q)} \\
\le  \inf_{f_h \in P(\calT_h)} (1 + \|\calJ_h\|_{\calL(L^p;L^p)}) \|f - f_h \|_{L^p(\Dom;\Real^q)} 
 \le c \inf_{f_h \in P(\calT_h)} \|f - f_h\|_{L^p(\Dom;\Real^q)},
\end{multline*}
which establishes \eqref{Th:Pih_item2}.  Let us now prove
\eqref{Th:Pih_item3}. We are just going to show that the leftmost square
commutes in the top diagram; the proof for the other squares is identical, and whether
boundary conditions are imposed or not is immaterial in the argument. Let
us first show that $J_h\upc \GRAD \phi_h = \GRAD (J_h\upg \phi_h)$ for 
all $\phi_h\in P\upg(\calT_h)$. We observe that
\begin{align*}
\GRAD \phi_h & = \GRAD (\inter_h\upg \calK\upg_{\delta|P\upg(\calT_h)} J_h\upg \phi_h)
= \GRAD (\inter_h\upg\calK\upg_\delta J_h\upg \phi_h) =
\inter\upc_h \GRAD (\calK\upg_\delta J_h\upg \phi_h) 
 = \inter\upc_h\calK\upc_\delta \GRAD (J_h\upg \phi_h),
\end{align*}
where we have used that
$\polI = \inter_h\upg \calK\upg_{\delta|P\upg(\calT_h)} J_h\upg$ (see
\eqref{def_of_Jh}), then
$\inter_h\upg \calK\upg_{\delta|P\upg(\calT_h)} J_h\upg= \inter_h\upg
\calK_\delta\upg J_h\upg$
(the range of $J_h\upg$ is in $P\upg(\calT_h)$), followed by
$\GRAD \inter_h\upg = \inter\upc_h \GRAD$ (see diagram~\eqref{Diag:V_Ph}) and
$\GRAD \calK\upg_\delta = \calK\upc_\delta \GRAD$ (see diagram~\eqref{Diag:W_C_Kdelta}). 
Since $\GRAD (J_h\upg \phi_h) \in \bP\upc(\calT_h)$ (see diagram~\eqref{Diag:V_Ph}), the above
argument together with \eqref{def_of_Jh} proves that 
\[
\GRAD \phi_h = (\IN_h\calK\upc_{\delta|\bP\upc(\calT_h)}) \GRAD (J_h\upg \phi_h) =
(J_h\upc)^{-1} \GRAD (J_h\upg \phi_h). 
\]
In conclusion, $J_h\upc \GRAD \phi_h = \GRAD (J_h\upg \phi_h)$.  Now we
finish the proof by using an arbitrary function $\phi\in
V\upg(\Dom)$. We have
\begin{align*}
\calJ_h\upc \GRAD \phi 
& = J_h\upc \IN_h\calK\upc_\delta \GRAD \phi 
= J_h\upc \IN_h \GRAD (\calK\upg_\delta\phi) 
= J_h\upc \GRAD (\inter_h\upg \calK\upg_\delta\phi)
= \GRAD (J_h\upg\inter_h\upg \calK\upg_\delta\phi).
\end{align*}
The last equality results from the fact that
$J_h\upc \GRAD \phi_h = \GRAD (J_h\upg \phi_h)$ for all
$\phi_h\in P\upg(\calT_h)$, as established above. This proves that
$\calJ_h\upc \GRAD \phi = \GRAD \calJ_h\upg\phi$. 
\end{proof}

\begin{remark}[Approximation]
Theorem~\ref{Th:Pih}(ii) shows that the quasi-interpolation error is bounded
by the best approximation error. Estimates of best approximation 
errors in fractional-order Sobolev spaces have been obtained 
recently in~\citet{ErnGuermond:15} for 
general finite element spaces.  
As an illustration, consider a $\bP\upc(\calT_h)$-based finite element 
approximation of a field $\bA \in \bZ\upcp(\Dom)$ (typically, with $p=2$). Suppose that the natural stability norm for this problem is that of $\Hrt$ and that the finite element solution $\bA_h\in \bP\upc(\calT_h)$ satisfies the a priori error estimate $\|\bA-\bA_h\|_{\Hrt} \le c \inf_{\ba_h\in \bP\upc(\calT_h)} \|\bA-\ba_h\|_{\Hrt}$. Then, taking $\ba_h=\calJ_h\upc\bA$ and using the commuting property leads to the bound
\[
\|\bA-\bA_h\|_{\Hrt} \le c (\|\bA-\calJ_h\upc\bA\|_{\bL^2(\Dom)} + 
\|\ROT\bA - \calJ_h\upd(\ROT\bA)\|_{\bL^2(\Dom)}).
\]
Assume that $\bA,\ROT\bA\in \bH^{r}(\Dom)$ for some real number
$r\in (0,k+1]$ where $k$ is the degree of the finite elements
composing $\bP\upc(\calT_h)$.  Then, using Theorem~\ref{Th:Pih}(ii)
together with~\cite[Cor.5.4]{ErnGuermond:15} leads to 
$\|\bA-\bA_h\|_{\Hrt} \le
ch^{r}(|\bA|_{\bH^r(\Dom)}+|\ROT\bA|_{\bH^r(\Dom)})$.
Note that no lower bound on $r$ is assumed a priori.
\end{remark}

\subsection{Discrete Poincar\'e inequalities} \label{SubSec:Poincare}
We illustrate the usefulness of the operators constructed above by 
proving discrete Poincar\'e inequalities for $\Hrt$-elements
in dimension $d=2,3$; we expose the material for $d=3$.
Assume also that $\Dom$ is
partitioned into $M$ connected, strongly Lipschitz subdomains
$\Dom_1,\cdots,\Dom_M$.  We consider two piecewise-smooth second-order tensor
fields $\bbeps$ and $\bbmu$, \ie we assume that these fields are in
\begin{equation}
  \polW^{1,\infty}(\cup_{i=1}^M\Dom_i) :=\left\{ \bbnu\in \polL^{\infty}(\Dom) \ | \
\GRAD(\bbnu_{|\Dom_i})\in [\polL^{\infty}(\Dom_i)]^{d},\; i=1,\cdots, M\right\},
\end{equation}
where $\polL^\infty(E):=L^{\infty}(E;\Real^{d\CROSS d})$. We additionally
assume that $\bbeps$ and $\bbmu$ are symmetric and the smallest
eigenvalue of each of these two tensors is bounded away from zero
from below uniformly over $D$.  Consider the following Maxwell
eigenvalue problems: Find $\bE$ and $0\ne \omega\in \polR$ such that
\begin{equation}
\ROT (\bbmu^{-1} \ROT \bE) = \omega \bE, 
\quad \DIV(\bbeps \bE)=0,
\quad \bE\CROSS \bn_{|\front}=\bzero, 
\label{Maxwell_eigenvalue_pb}
\end{equation}
Find $\bB$ and $0\ne \omega\in \polR$ such that
\begin{equation}
\ROT (\bbmu^{-1} \ROT \bB) = \omega \bB, 
\quad \DIV(\bbeps \bB)=0,
\quad (\bbeps\bB)\SCAL\bn_{|\front}=0. 
\label{Maxwell_eigenvalue_pb_B}
\end{equation}
Upon setting $\bH_{\CROSS\bn} := \{\bz\in \Hrt\st \DIV(\bbeps \bz)=0, \ \bz\CROSS \bn_{|\front}=\bzero\}$, 
$\bH_{\SCAL\bn} := \{\bz\in \Hrt\st \DIV(\bbeps \bz)=0, \ (\bbeps\bz)\SCAL\bn_{|\front}=0\}$, the
$L^2$-theory of the well-posedness of this problem is based on the
following embedding inequality: There are $c>0$ and $s>0$ (both
depending on $\Dom$ and $\bbeps$) such that
\begin{subequations}\begin{align}
\|\be\|_{\bH^s(\Dom)} \le c\, \|\ROT \be\|_{\bL^2(\Dom)}, \quad  \forall \be \in \bH_{\CROSS\bn},\label{Hs_Embedding}\\
\|\bb\|_{\bH^s(\Dom)} \le c\, \|\ROT \bb\|_{\bL^2(\Dom)}, \quad  \forall \bb \in \bH_{\SCAL\bn},\label{Hs_Embedding_B}
\end{align}\end{subequations}
provided $\front$ is connected and $\Dom$ is simply connected, respectively.
The above inequalities, proved in
\citet{Bonito_Guermond_Luddens_2013}, generalize classical
inequalities established by \citet{Costabel_1990} and
\citet{Birman_Solomyak_1987} assuming that the tensor $\bbeps$ is smooth over
the entire domain.

Let us consider the finite element approximation of the above
eigenvalue problem using the setting described in the previous
sections. The approximation theory of this problem is non-trivial,
especially when using finite elements that do not fit the De Rham
Diagram.  We refer to the book of \citet{Monk_book_2003} and the review
by \citet{Hiptmair_acta_2002} for an overview on the topic.  

Let $\bP\upc(\calT_h)$, $\bP_0\upc(\calT_h)$ be defined as above.  A
key step for approximating \eqref{Maxwell_eigenvalue_pb}
or~\eqref{Maxwell_eigenvalue_pb_B} consists of establishing the
following discrete Poincar\'e inequalities: There is $c>0$, uniform
with respect to $h$, such that
\begin{subequations}\label{eq:disc_PoinK}\begin{align}
\|\be_h\|_{\bL^2(\Dom)} \le c \, \|\ROT \be_h\|_{\bL^2(\Dom)},\quad \forall \be_h\in \bH_{h,\CROSS\bn},
\label{Poincare_Curl}\\
\|\bb_h\|_{\bL^2(\Dom)} \le c \, \|\ROT \bb_h\|_{\bL^2(\Dom)},\quad \forall \bb_h\in \bH_{h,\SCAL\bn},
\label{Poincare_Curl_B}
\end{align}\end{subequations}
where
$\bH_{h,\CROSS \bn}:= \{\bv_h\in \bP_0\upc(\calT_h)\st \int_\Dom (\bbeps \bv_h)
\SCAL \GRAD q_h \dif x =0, \ \forall q_h \in P\upg_{0}(\calT_h)\}$
and
$\bH_{h,\SCAL\bn}:= \{\bv_h\in \bP\upc(\calT_h)\st \int_\Dom (\bbeps \bv_h)
\SCAL \GRAD q_h \dif x =0, \ \forall q_h \in P\upg(\calT_h)\}$.
There are many ways of proving
\eqref{Poincare_Curl}-\eqref{Poincare_Curl_B} when $\bbeps$ is smooth,
since in this case it can be proved that the Sobolev index $s$ in
\eqref{Hs_Embedding} is larger than $\frac12$. The first route
described in \cite[\S4.2]{Hiptmair_acta_2002} consists of invoking
subtle regularity estimates from
\citet[Lemma~4.7]{Amrouche_Bernardi_Dauge_Girault_1998}. The second
one, which avoids invoking regularity estimates, is based on the so-called discrete compactness argument of
\citet{Kikuchi_1989} and further developed by
\citet{Monk_Demkowicz_2001} and
\citet{Caorsi_Fernades_Raffetto_2000}. The proof is not constructive
and is based on an argument by contradiction.

We now show that using the approximation operators described in the
previous sections gives a direct answer to the above question without
requiring any particular condition on the Sobolev index $s$ in
\eqref{Hs_Embedding}; see also \citet[Thm~3.6]{Arnold_Falk_Winther_2010}.
\begin{theorem}[Discrete Poincar\'e] \label{Thm:Poincare} Assume that $\front$ is connected (resp., $\Dom$ is simply connected). Then there is a uniform constant $c>0$
  such that \eqref{Poincare_Curl} holds (resp., \eqref{Poincare_Curl_B} holds).
\end{theorem}
\begin{proof}
  We only do the proof for~\eqref{Poincare_Curl}, the proof for~\eqref{Poincare_Curl_B} is similar.
Let $\bv_h\in\bH_{h,\CROSS\bn}$ be a nonzero discrete field.
Let $\phi(\bv_h)\in \Hunz$ be the solution to the following Poisson problem:
\[
\DIV(\bbeps\GRAD\phi(\bv_h)) = \DIV(\bbeps\bv_h),\qquad \phi(\bv_h)_{|\front} =0.
\]
Note that this problem is well-posed owing to the assumed regularity
and structure of $\bbeps$. Let us define
$\bv(\bv_h):= \bv_h - \GRAD\phi(\bv_h)$.  This definition implies that
\[
\DIV(\bbeps\bv(\bv_h)) = 0,\quad \ROT (\bv(\bv_h))= \ROT \bv_h, \quad
\bv(\bv_h)\CROSS\bn_{|\front}= \bzero,
\] 
so that $\bv(\bv_h)\in \bH_{\CROSS\bn}$. 
We now bound $\|\bv_h\|_{\bL^2(\Dom)}$ as follows:
\begin{align*}
c \|\bv_h\|_{\bL^2(\Dom)}^2 & \le \int_\Dom(\bbeps\bv_h)\SCAL \bv_h \dif x
= \int_\Dom(\bbeps\bv_h)\SCAL (\bv_h - \calJ_{h0}\upc\bv(\bv_h) + \calJ_{h0}\upc\bv(\bv_h))  \dif x\\
& = \int_\Dom(\bbeps\bv_h)\SCAL \calJ_{h0}\upc(\bv_h - \bv(\bv_h))\dif x 
+ \int_\Dom(\bbeps\bv_h)\SCAL\calJ_{h0}\upc\bv(\bv_h)  \dif x \\
& = \int_\Dom(\bbeps\bv_h)\SCAL \calJ_{h0}\upc\GRAD(\phi(\bv_h))\dif x 
+ \int_\Dom(\bbeps\bv_h)\SCAL\calJ_{h0}\upc\bv(\bv_h)  \dif x.
\end{align*}
Note here that we used that $\calJ_{h0}\upc \bv_h = \bv_h$.
Then using the commuting property $\calJ_{h0}\upc\GRAD(\phi(\bv_h)) = \GRAD(\calJ_{h,0}\upg\phi(\bv_h))$ and since $\calJ_{h,0}\upg$ maps onto $P\upg_0(\calT_h)$,
we infer that
\begin{align*}
c \|\bv_h\|_{\bL^2(\Dom)}^2 &\le  \int_\Dom(\bbeps\bv_h)\SCAL \GRAD(\calJ_{h,0}\upg\phi(\bv_h))\dif x 
+ \int_\Dom(\bbeps\bv_h)\SCAL\calJ_{h0}\upc\bv(\bv_h)  \dif x  \\
& = \int_\Dom(\bbeps\bv_h)\SCAL\calJ_{h0}\upc\bv(\bv_h)  \dif x \le
c' \|\bv_h\|_{\bL^2(\Dom)} \|\calJ_{h0}\upc\bv(\bv_h) \|_{\bL^2(\Dom)}.
\end{align*}
The uniform boundedness of $\calJ_{h0}\upc$ on $\bL^2(\Dom)$ and \eqref{Hs_Embedding} with $s=0$ imply 
\begin{align*}
\|\bv_h\|_{\bL^2(\Dom)} & \le c \|\calJ_{h0}\upc\bv(\bv_h) \|_{\bL^2(\Dom)} 
\le c' \|\bv(\bv_h) \|_{\bL^2(\Dom)} 
\le c'' \|\ROT\bv_h\|_{\bL^2(\Dom)}.
\end{align*}
This concludes the proof.
\end{proof}

\section*{Acknowledgments} The authors acknowledge fruitful
discussions with S.~H. Christiansen and A. Demlow.

\bibliographystyle{abbrvnat} 
\bibliography{ref}

\end{document}